\documentclass[11pt]{article}
\usepackage[sort,numbers]{natbib}
\usepackage{amsmath}
\usepackage{amsfonts}
\usepackage{amssymb}
\usepackage{epsfig}
\usepackage[font=normalsize]{subfig}
\usepackage{amscd}
\usepackage{latexsym}
\usepackage[colorlinks=black,linkcolor=red,anchorcolor=blue,citecolor=green]{hyperref}
\usepackage{CJK}
\usepackage{color}
\usepackage{leftidx}

\newtheorem{theorem}{\bf Theorem}[section]
\newtheorem{lemma}[theorem]{\bf Lemma}
\newtheorem{proposition}[theorem]{\bf Proposition}

\newenvironment{proof}{\noindent{\em Proof:}}{\quad \hfill$\Box$\vspace{2ex}}

\makeatletter
\newcommand{\figcaption}{\def\@captype{figure}\caption}
\newcommand{\tabcaption}{\def\@captype{table}\caption}
\makeatother

\def \aN {\Bbb N}
\def \aC {\Bbb C}
\def \aZ {\Bbb Z}
\def \aR {\Bbb R}

\begin{document}
%\begin{CJK}{GBK}{song}

\title{\bf Numerical analysis for the moments of highly oscillatory Bessel functions and Bessel-trigonometric functions}%Study Report 8
\author{ Yinkun Wang
    \footnotemark[2]
    \footnotemark[3],
    Ying Li
    \footnotemark[2],
and
    Jianshu Luo
    \footnotemark[2]}

\renewcommand{\thefootnote}{\fnsymbol{footnote}}

\footnotetext[2]{College of Science, National University of Defense Technology, Changsha 410073, People's Republic of China. }
\footnotetext[3]{Corresponding author. Email address: yinkun5522@163.com.}

%\date{6/12/2015}
\maketitle{}

\begin{abstract}
The moments of highly oscillatory Bessel functions and Bessel-trigonometric functions play a basic role in many practical problems and numerical analysis. This paper presents a complete analysis for these moments based on the recursive relations of Bessel functions. To evaluate the moments of Bessel functions numerically, a fast and efficient scheme is also proposed to approximate the integral of Bessel function of the first kind and of zero order. The moments of Bessel-trigonometric functions are proved to be expressed in a closed form. In the numerical results, the accuracy and efficiency of the proposed analysis for the moments of Bessel functions is validated first and then by comparing the existing methods, a better scheme for the moments of Bessel functions is presented.
\end{abstract}


\section{Introduction}

The integrations containing Bessel functions play an important role in many practical problems in physics, chemistry and engineering. There have been developed lots of efficient numerical quadratures in evaluating these integrations. One kind of commonly used quadratures is the Filon method and its further development, Filon-type method. It is well-known that the Filon method and Filon-type method require that the corresponding the moments of Bessel functions are known or can be evaluated fast and accurately. Moreover, the moments of Bessel functions  often forms the linear systems in the solution of integral equations with Bessel kernel by collocation methods with a polynomial or piecewise polynomial base. Therefore, the explicit expressions or fast efficient numerical analysis of related moments of Bessel functions is urgently required in the science and engineering.

The focus of this paper is to present efficient numerical analysis for the corresponding moments of Bessel functions which have the form
\[I_1(n,m,\kappa,b):= \int_0^b t^nJ_{m}(\kappa t)dt\]
and
\[I_2(n,m,\kappa,b):=\int_0^b t^ne^{ i\kappa t}J_{m}(\kappa t)dt\]
where $b,\kappa\in \aR$, $|b|\leq1$, $n,m\in\aN_0:=\aN\cup\{0\}$ and $J_m$ is the Bessel function of the first kind and order $m$.
%These two kinds of integrals are the basic moments required by the Filon and Filon-type methods in calculating the highly oscillatory integrals related to the Bessel functions.
%Therefore, the explicit expressions or fast efficient numerical analysis of these two kinds are vital to operate the efficient Filon and Filon-type methods in the applications with respect to Bessel functions.

The first integral, $I_1$, has already been used in the applications and there exist some explicit formulas in the form of infinite series or in the form of other special functions. They are given as follows (see for example in \citep{LUKE1962} pp. 44, 51-52, 85, \citep{STEGUN} pp. 480, and \citep{BATEMAN1954} pp. 22, \citep{WATSON1952}, pp. 350):

\quad\quad\quad $I_1(n,m,\kappa,b)$
\begin{align}
     &=\frac{\kappa^mb^{n+m+1}}{2^m(n+m+1)\Gamma(m+1)}{_1}F_2\left( \frac{n+m+1}{2},\frac{n+m+3}{2};m+1;-\frac{\kappa^2 b^2}{4}\right) \label{n6e21}\\
    & =\frac{b^n}{\kappa (n+m+1)}\sum_{j=0}^\infty (2j+m+1)\frac{m+1-n}{m+3+n}\ldots \frac{m+2j-1-n}{m+2j+1+n}J_{2j+m+1}(\kappa b) \label{n6e22} \\
    &=\frac{2^n\Gamma\left(\frac{m+n+1}{2}\right)}{\kappa^{n+1}\Gamma\left(\frac{m-n+1}{2}\right)} +\frac{b}{\kappa^n} \left[(n+m-1)J_{m} (\kappa b) s_{n-1,m-1}^{(2)}(\kappa b) -J_{m-1}(\kappa b) s_{n,m}^{(2)}(\kappa b) \right] \label{n6e23}
\end{align}
where $\Gamma(t)$ is the gamma function, ${_1}F_2(n,m;\kappa;t)$ is the Gaussian hypergeometric function, $s_{n,m}^{(2)}(t)$ denotes the Lommel function of the second kind and $\kappa b>0$ in (\ref{n6e23}). The Lommel function in (\ref{n6e23}) has an asymptotic expansion for $t$ large enough (see for example in \citep{WATSON1952} pp. 351-352):
\begin{equation}\label{n6e26}
  s_{n,m}^{(2)}(t)=t^{n-1}\left[1-\frac{(n-1)^2-m^2}{t^2}+\frac{\left[(n-1)^2-m^2\right] \left[(n-3)^2-m^2\right]}{t^4}-\ldots \right].
\end{equation}
The formulas (\ref{n6e22}) and (\ref{n6e23}) have been used in approximating highly oscillatory Bessel transforms or solving numerically the Volterra integral equations of the second kind with highly oscillatory Bessel functions \citep{XIANG2011,XIANG2013,FANG2015,CHEN2015}. Specially, two complete numerical schemes were proposed in \citep{XIANG2011,XIANG2013} for the evaluation of $I_1(n,m,\kappa,b)$. Both of two schemes adopted the formulas (\ref{n6e22}) and (\ref{n6e26}) with a little difference in the choose of the parameters. We state the latter one in \citep{XIANG2013}.
When $|\kappa b|<50$, the integral $I_1(n,m,\kappa,b)$ is approximated by the first 60 truncated terms of (\ref{n6e22}), otherwise, it is estimated by the formula (\ref{n6e23}) with the Lommel function approximated by the first 10 truncated terms of (\ref{n6e26}). However, this scheme perform less efficiently for small $n$ and $m$ when $|\kappa b|<50$ since 60 values of Bessel functions need to be evaluated and it may not work well for the case $\kappa b<-50$ since the formula (\ref{n6e23}) is only true for $\kappa b>0$.  This motivate us to find another way to evaluate the corresponding moments fast and accurately.

The second moment, $I_2(n,m,\kappa,b)$, has less applications currently. It also has some explicit formulas in the form of infinite series or in the form of other special functions, (see for example in \citep{LUKE1962} pp. 95).

\quad\quad\quad $I_2(n,m,\kappa,b)$
\begin{align}
     &=\frac{(2\kappa)^m b^{n+m+1}}{\Gamma(1/2)}\sum_{j=0}^\infty \frac{\Gamma(m+j+1/2)(2i\kappa b)^j}{\Gamma(2m+j+1)j!(n+m+j+1)} \label{n6e24}\\
     &=\frac{(\kappa/2)^mb^{n+m+1}}{(n+m+1)\Gamma(m+1)} {_2}F_2(m+1/2,n+m+1;2m+1,n+m+2;2i\kappa b) \label{n6e25}
\end{align}
However, the formulas (\ref{n6e24}) and (\ref{n6e25}) are not suitable in the numerical analysis of $I_2(n,m,\kappa,b)$. It is because that formula (\ref{n6e24}) needs a large number of truncated terms to obtain the accepted accuracy especially for large $\kappa b$ while formula  (\ref{n6e25}) has to analyze the Gaussian hypergeometric function.

Besides, some efforts have been drawn on the study of the Levin method \citep{LEVIN1996,XIANG2008a}, the Levin-type method \citep{OLVER2006} and the generalized quadrature rule \citep{EVANS2000,XIANG2008b} in evaluating the integrations containing Bessel functions.
When $\kappa$ is large, the moments $I_1(n,m,\kappa,b)$ and $I_2(n,m,\kappa,b)$ are highly oscillatory integrals and may be evaluated numerically by these methods for some cases. However, there is a disadvantage that these approaches only work for the case $0\not \in [a,b]$ since they need use the differential relations of Bessel functions. Hence, these methods may fail or give rise to large errors when $a$ or $b$ is close to or equals the origin.

The main idea in calculating the moments $I_1(n,m,\kappa,b)$ and $I_2(n,m,\kappa,b)$ is to deduce stable recursive iterations with respect to large $\kappa$ to transform these moments to the integrals which have the explicit formulas or can be analyzed by numerical methods fast and accurately.  For $I_1(n,m,\kappa,b)$, a well-known result is that half of them can be given explicit in finite  terms and the other half depends on the integral $I_1(0,0,\kappa,b)$. To evaluate $I_1(n,m,\kappa,b)$, we present an efficient scheme in evaluation of $I_1(0,0,\kappa,b)$ by combining the trapezoidal rule and the numerical steepest methods. We discover some useful and stable recursive relations and then present explicit formulas for different cases of $I_1(n,m,\kappa,b)$. For $I_2(n,m,\kappa,b)$, it is found that they all can be analyzed explicitly in finite terms. By finding some useful recursive iterations, we present the explicit formulas for different cases of them.

This paper is organized as follows. In Section 2 the evaluation is presented for the moments $I_1(n,m,\kappa,b)$. Thereafter, the moments $I_2(n,m,\kappa,b)$ is discussed in Section 3. We then present in Section 4 several numerical results to validate the efficiency of the proposed method for $I_1(n,m,\kappa,b)$ in Section 2  and then by comparing with the existing methods, a better scheme for $I_1(n,m,\kappa,b)$ is presented.

\section{Evaluation of $I_1(n,m,\kappa,b)$}
In this section, we present the evaluation of $I_1(n,m,\kappa,b)$. It is well-known in \citep{ANDREWS1985} that $I_1(n,m,\kappa,b)$ can be integrated in closed form when $m+n$ is odd but ultimately depends upon the integral $I_1(0,0,\kappa,b)$ which cannot be evaluated in closed form when $m+n$ is even. We first give an efficient numerical scheme to evaluate $I_1(0,0,\kappa,b)$ and then derive some necessary recursive relationships of $I_1(n,m,\kappa,b)$ for numerical purpose.
\subsection{Evaluation of $I_1(0,0,\kappa,b)$}
We present in this subsection an efficient scheme in evaluation of $I_1(0,0,\kappa,b)$ reaching the machine tolerance. The integral is analyzed numerical by the combination of trapezoidal rule and the numerical steepest method and its corresponding error analysis is presented. The integral $I_1(0,0,\kappa,b)$ is denoted as $I_1(\kappa,b)$ for notation simplicity in this subsection.

Noting that $J_0(t)=\frac{1}{2\pi}\int_{-\pi}^\pi e^{-i t\sin \phi}d \phi$, we substitute $J_0(\kappa t)$ in the integral $I_1(\kappa,b)$ with  the integral expression and $I_1(\kappa,b)$  is reformed by exchanging the integral orders as
\begin{equation*}
  I_1(\kappa,b)=\frac{1}{2\pi}\int_{-\pi}^{\pi}\frac{e^{-i\kappa b \sin \phi}-1}{-i\kappa \sin\phi} d\phi.
\end{equation*}
With some calculation, the above integral expression can be simplified as
\begin{equation}\label{n6e1}
  I_1(\kappa,b)=\frac{1}{2\pi}\int_0^{\pi} \frac{e^{i\kappa b\sin\phi}-e^{-i\kappa b\sin\phi}}{i\kappa \sin\phi}d\phi
\end{equation}
or
\begin{equation}\label{n6e2}
  I_1(\kappa,b) = \frac{1}{\pi}\int_0^{\pi/2} \frac{e^{i\kappa b\sin\phi}-e^{-i\kappa b\sin\phi}}{i\kappa \sin\phi}d\phi.
\end{equation}
The equations (\ref{n6e1}) and (\ref{n6e2}) are the two main formulas to develop the efficient numerical method for the evaluation of $I_1(\kappa,b)$.
Since the integrand in (\ref{n6e1}) can be written as the convergent power series, it clear that it is analytic and $\pi$-periodic. Therefore it is efficient to evaluate the integral by the trapezoidal integral rule when $\kappa b$ is not large. We next present the error analysis for the trapezoidal integral rule in evaluating the integral in (\ref{n6e1}) and will derive the dependence of the error on the parameters. To this end, we recall a well-known result on the trapezoidal integral rule \citep{KRESS1998}.
\begin{lemma}\label{n6l1}
Let  $f:\aR\rightarrow\aR$ be analytic and $2\pi$-periodic. Then there exists a strip $D=\aR\times(-a,a)\subset \aC$ with $a>0$ such that $f$ can be extended to a holomorphic and $2\pi$-periodic bounded function $f: D\rightarrow \aC$. The error for the rectangular rule with $N$ points can be estimated by
\[
|E_N|\leq \frac{4\pi M}{e^{Na}-1},
\]
where $M$ denotes a bound for the holomorphic function $f$ on $D$.
\end{lemma}

Note that when the integrand is periodic, the trapezoidal integral rule is the same as the rectangular rule.

\begin{proposition}
  If $I_1(\kappa,b)$ in (\ref{n6e1}) is evaluated by using the trapezoidal integral rule with $N$ points, then for any given $c>0$ the error for $I_1(\kappa,b)$ can be bounded by
  \[
    |E_N|\leq \frac{8\pi}{1- e^{-Nc}}e^{\kappa b C-Nc},
  \]
  where $C=\frac{e^{-c/2}+e^{c/2}}{2}$.
\end{proposition}
\begin{proof}
The proof is the direct application of Lemma \ref{n6l1}. We only need derive the bound for the integrand, denoted by $f$,  in (\ref{n6e1}). Since $f$ is $\pi$-periodic, we make a change of variable $\phi=2\phi$ such that $f$ is $2\pi$-periodic and obtain a series expression for $f$,
\[
f(\phi)=2\sum_{j=0}^\infty \frac{(i\kappa b)^{2j+1}}{(2j+1)!}\sin^{2n}\frac{\phi}{2}.
\]
Suppose that $c$ is a given positive number. For $\phi\in\aC$ and Imag$(\phi)\leq c$, we have the inequality that
\[
|\sin \phi|\leq \frac{e^{-c}+e^{c}}{2}
\]
Let $C:=\frac{e^{-c/2}+e^{c/2}}{2}$ and thus $C\geq 1$. We shall get the bound for $f$ when $\phi\in\aC$ and Imag$(\phi)\leq c$,
\[
|f|\leq 2\sum_{j=0}^\infty \frac{(\kappa b)^{2j+1}}{(2j+1)!}C^{2n}\leq 2\sum_{j=0}^\infty \frac{(\kappa bC)^{2j+1}}{(2j+1)!} \leq 2e^{\kappa bC}
\]
Hence the error bound follows directly with the help of Lemma \ref{n6l1}.
\end{proof}

We take the approximate value of $c$ which minimized $C/c$ and get an approximate error bound for $I_1(\kappa,b)$ which is given by $8\pi e^{2.4(0.75\kappa b-N)}$. It is obvious that the number $N$ shall increase linearly as $\kappa b$ to attain the accuracy. It cost much in computation by the trapezoidal integral rule and, therefore, we must apply other methods while $\kappa b$ is very large.

 Next, we adopt the numerical steepest method \citep{HUYBRECHES2006} to evaluate $I_1(\kappa,b)$ through (\ref{n6e2}). Without the loss of generality, we assume that $\kappa b\geq 0$, otherwise we may consider $I_1(\kappa,-b)$ which equals $-I_1(\kappa,b)$ when $b<0$ or $I_1(-\kappa,b)$ which equals $I_1(\kappa,b)$ when $\kappa<0$ according to equation (\ref{n6e2}). We need use the numerical steepest method for each exponential function in the integrand since there is no route in complex plane such that two conjugate functions decays exponentially at the same time. We also note that each integral is a divergent improper integral when each exponential function is handled separately. However, it will not be a trouble since we use the numerical steepest method to each part in form and the divergence in each part will cancel out when adding two integral together again. We then handle one of them, $I_{1,1}(\kappa,b):=\frac{1}{\pi}\int_0^{\pi/2} \frac{e^{i\kappa b\sin\phi}}{i\kappa \sin\phi}d\phi$, and the other one can be done similarly. The key ideal of the numerical steepest method is to choose the proper integration routes and then to use the Cauchy integral theorem. According to the instruction of the numerical steepest method, the routes on which the integrand in $I_{1,1}(\kappa,b)$ is non-oscillatory and decays exponentially can be chosen as
 $h_{0}(p)=\arcsin (ip)$ and $h_{\pi/2}(p)=\arcsin(1+ip)$ where $p\geq0$. To use the Cauchy integral theorem, we must scoop out the origin point since the integrand in $I_{1,1}(\kappa,b)$ is singular there. Therefore, we introduce a one-quarter circle route $\Gamma$ whose radius is $\varepsilon>0$ around the origin and the domain for $h_0$ is $p\geq \varepsilon$. The final integration routes are illustrated in Figure \ref{n6f1}.
\begin{figure}
  \centering
  \includegraphics[width=3in]{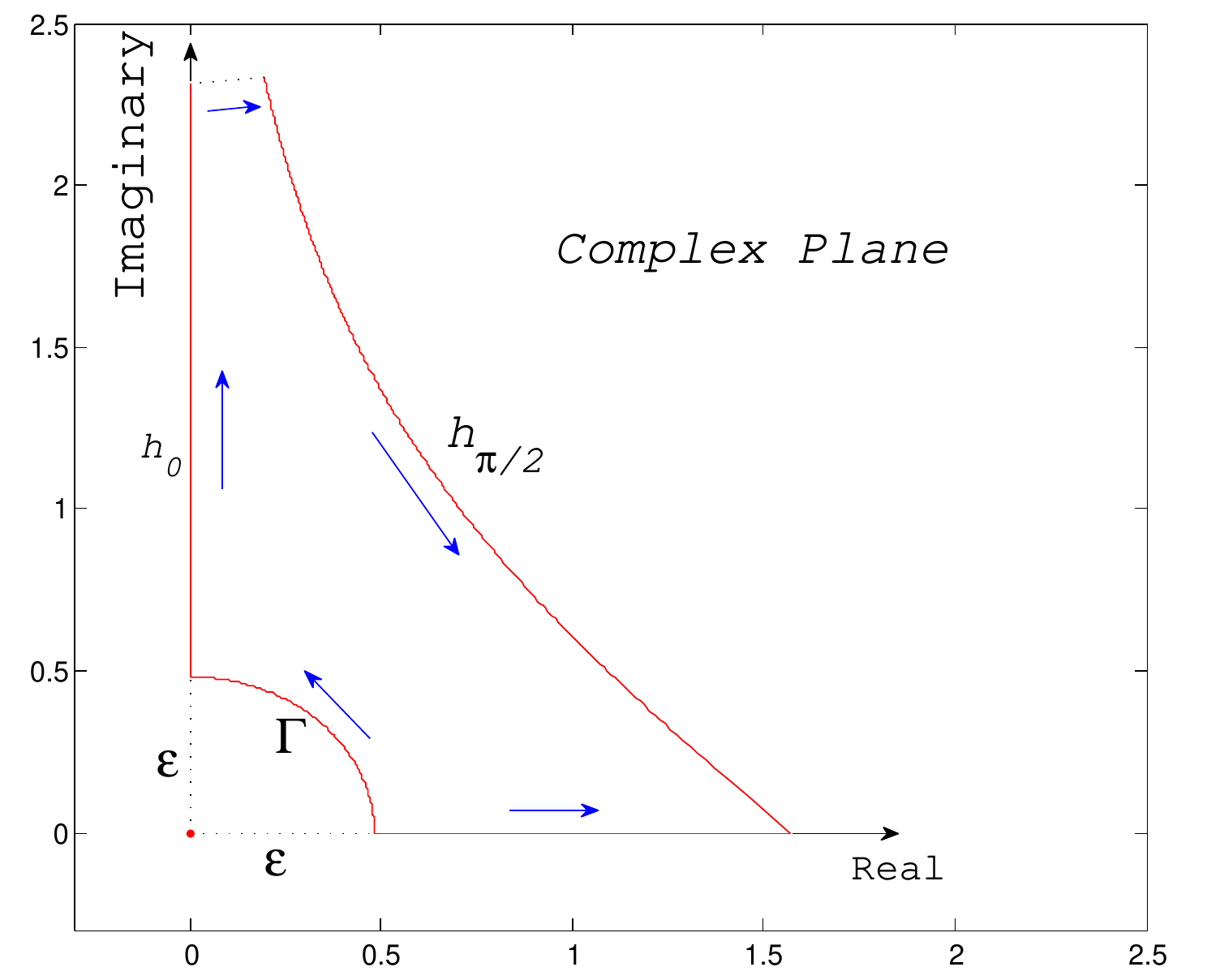}\\
  \caption{Illustration of the integration routes for $I_{1,1}(\kappa, b)$}\label{n6f1}
\end{figure}
Then by using the Cauchy integral theorem, we have with some calculation that
\[
\begin{split}
I_{1,1}(\kappa,b)& =\lim_{\varepsilon\rightarrow 0}\frac{1}{\pi}\left(\int_\Gamma+\int_{h_0}+\int_{h_{\pi/2}} \right) \frac{e^{i\kappa b\sin\phi}}{i\kappa \sin\phi}d\phi \\
&=\frac{1}{2\kappa}+\frac{1}{\kappa\pi}\lim_{\varepsilon\rightarrow 0}\int_\varepsilon^{+\infty} \frac{e^{-\kappa bp}}{ip\sqrt{1+p^2}}dp-\frac{1}{\kappa\pi} \int_0^{+\infty} \frac{e^{i\kappa b}e^{-\kappa bp}}{(1+ip)\sqrt{p(p-2i)}}dp.
\end{split}
\]
The first integral in the equation above is carried out by making a polar transformation while the other two are handled normally by making change of variables. It can be seen that the second integral is divergent which will be canceled out by the corresponding part in the $I_{1,2}(\kappa,b):=\frac{1}{\pi}\int_0^{\pi/2} \frac{e^{-i\kappa b\sin\phi}}{i\kappa \sin\phi}d\phi$. Undergoing the same way as $I_{1,1}(\kappa,b)$, we present directly the formula for $I_{1,2}(\kappa,b)$,
 \[
I_{1,2}(\kappa,b)=-\frac{1}{2\kappa}+\frac{1}{\kappa\pi}\lim_{\varepsilon\rightarrow 0}\int_\varepsilon^{+\infty} \frac{e^{-\kappa bp}}{ip\sqrt{1+p^2}}dp+\frac{1}{\kappa\pi} \int_0^{+\infty} \frac{e^{-i\kappa b}e^{-\kappa bp}}{(1-ip)\sqrt{p(p+2i)}}dp.
\]
Combining the equations for $I_{1,1}(\kappa,b)$ and $I_{1,2}(\kappa,b)$ together and then making a change of variable $p=\kappa bp$, we get that
\begin{equation}\label{n6e3}
\begin{split}
I_{1}(\kappa,b)&=\frac{1}{\kappa}-\frac{1}{\kappa\pi}\int_0^{+\infty}\frac{e^{-\kappa bp}}{\sqrt{p}}\left(\frac{e^{i\kappa b}}{(1+ip)\sqrt{p-2i}} +\frac{e^{-i\kappa b}}{(1-ip)\sqrt{p+2i}} \right)dp \\
%&=\frac{1}{\kappa}-\frac{b}{\pi}\int_0^{+\infty}\frac{e^{-p}}{\sqrt{p}}\left(\frac{e^{i\kappa b}}{(\kappa b+ip)\sqrt{p-2i\kappa b}} +\frac{e^{-i\kappa b}}{(\kappa b-ip)\sqrt{p+2i\kappa b}} \right)dp \\
&=\frac{1}{\kappa}-\frac{2b}{\pi}\int_0^{+\infty}p^{-1/2}e^{-p} Re\left(e^{i\kappa b}(\kappa b+ip)^{-1}(p-2i\kappa b)^{-1/2}\right) dp
\end{split}
\end{equation}
The numerical steepest method in evaluating $I_{1}(\kappa,b)$ is the formula in (\ref{n6e3}) in which the improper integral is calculated numerically by the generalized Gauss-Laguerre quadrature. We next present the error bound for the numerical steepest method in evaluating $I_{1}(\kappa,b)$ and show the dependence on the parameters. To this end, we recall the famous generalized Gauss-Laguerre formula \citep{DAVIS1984}.
If $f$ is $2N$-times continuously differentiable and $\alpha>-1$, then
\begin{equation}\label{n6e4}
  \int_0^\infty t^\alpha e^{-t}f(t)dt=\sum_{j=1}^N\omega_jf(t_j)+\frac{N!\Gamma(N+\alpha+1)}{(2N)!}f^{(2N)}(\xi),\ 0<\xi<\infty,
\end{equation}
where the abscissas $t_j$ are the zeros of the generalized Laguerre polynomial $L_N^{(\alpha)}(t)$ and the weights
\[
\omega_j=\frac{\Gamma(N+\alpha+1)t_{j}}{N!\left[L_{N+1}^{(\alpha)}(t_j) \right]^2}.
\]

\begin{proposition}
  If $I_1(\kappa,b)$ in (\ref{n6e3}) is evaluated by using the generalized Gauss-Laguerre quadrature with $N$ points, then the error for $I_1(\kappa,b)$ can be bounded by
  \[
    |E_N|\leq \frac{2^{3/2}b}{\pi}N!\Gamma(N+1/2)(\kappa b)^{-2N-3/2},\ N\in\aN.
  \]
\end{proposition}
\begin{proof}
Let $\alpha=-1/2$ and
\[
f(t)=2b Re\left(e^{i\kappa b}(\kappa b+it)^{-1}(t-2i\kappa b)^{-1/2}\right)/\pi.
\]
It is known that $f$ is infinitely differentiable on $[0,\infty)$. By using the Leibniz rule for the higher derivatives of a product of two factors, the derivative of order $2N$ of $f$ is given by
\[
f^{(2N)}(t)=\frac{2b}{\pi} Re\left(e^{i\kappa b}\sum_{j=0}^{2N}C_{2N}^j(-i)^jj!(\kappa b+it)^{-(j+1)} (-2)^{-(2N-j)}(4N-2j-1)!!(t-2i\kappa b)^{-(2N-j)-1/2}\right).
\]
We admit that $(-1)!!=1$ in the above equation. Then we have that for $t\geq0$,
\[
\begin{split}
|f^{(2N)}(t)| & \leq \frac{2b}{\pi} \sum_{j=0}^{2N}C_{2N}^j(\kappa b)^{-(j+1)} 2^{-(2N-j)}(4N-2j-1)!!(2\kappa b)^{-(2N-j)-1/2}\\
& \leq (\kappa b)^{-2N-3/2}(2N)!\frac{\sqrt{2}b}{\pi}\sum_{j=0}^{2N}\frac{(2j-1)!}{j!4^j} \\
& \leq \frac{2^{3/2}b}{\pi}(2N)!(\kappa b)^{-2N-3/2}
\end{split}
\]
With the bound of $f^{(2N)}$, the desired error bound follows directly from (\ref{n6e4}).
\end{proof}

We note that the error bound for the numerical steepest method with $N$-point generalized Gauss-Laguerre quadrature decrease as $N$ increase when $N\leq [\kappa b]$. However, it is better to use relatively small value of $N$ compared to $\kappa b$ since the weights can be extremely small and we may hardly obtain them with required accuracy.

We may find out an efficient scheme in evaluating $I_1(\kappa, b)$ with at least a machine tolerance for all $\kappa b$ with the help of the error bounds. For this purpose, we present a figure about the relation between $N$ and $\kappa b$, shown in Figure \ref{n6f2}, when the error bound is under the machine tolerance. For the generalized Guass Laguerre quadrature, we assume here that $b\leq 1$. According to the comparison in Figure \ref{n6f2}, we present a scheme for all $\kappa b$ to calculate $I_1(\kappa, b)$: when $\kappa b<24$, we adopt the trapezoidal integral rule with 36 points, otherwise, we choose the numerical steepest method with the 10-point generalized Gauss Laguerre quadrature.
\begin{figure}
  \centering
  \includegraphics[width=3.8in]{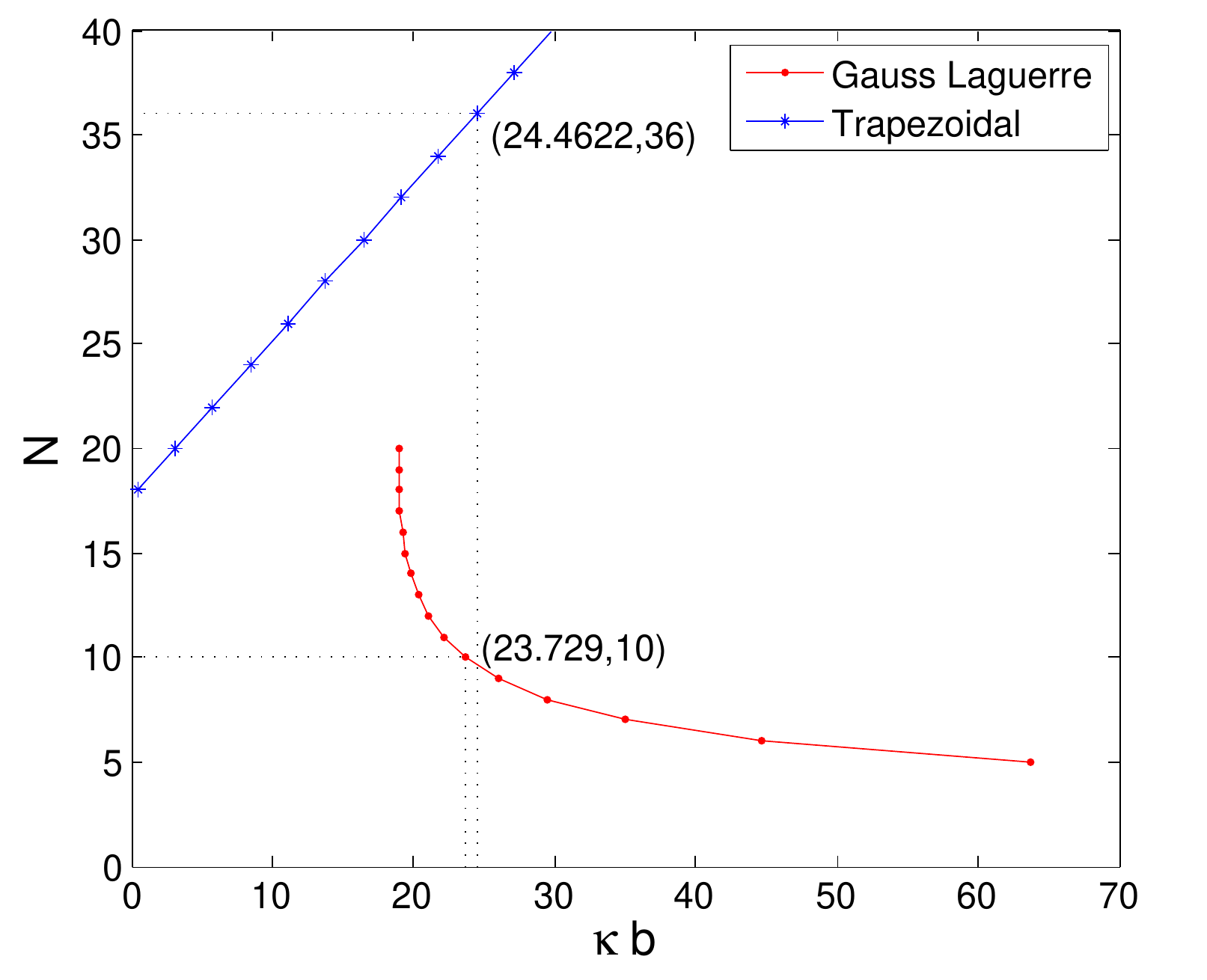}\\
  \caption{Relation between $N$ and $\kappa b$ when the error bound is under the machine tolerance.}\label{n6f2}
\end{figure}

\subsection{Evaluation of $I_{1}(n,m,\kappa,b)$}
In this subsection, we shall present the complete recursive relations of the integrals $I_{1}(n,m,\kappa,b)$. We denote the corresponding indefinite integrals by $I_1(n,m,\kappa,t):=\int t^nJ_{m}(\kappa t)dt$.

For this purpose, we recall some well-known results about the differentiation relations of Bessel function of the first kind \citep{ANDREWS1985} which are listed below
\begin{eqnarray}
  J_m'(t) &=& -\frac{m}{t}J_m(t)+J_{m-1}(t)  \\
  J_m'(t) &=& -J_{m+1}(t)+\frac{m}{t}J_m(t) \\
  \frac{d}{dt}\left[t^mJ_m(t)\right] &=& t^mJ_{m-1}(t) \\
  \frac{d}{dt}\left[t^{-m}J_m(t)\right] &=& -t^{-m}J_{m+1}(t)
\end{eqnarray}
Denote $J_{m,\kappa}(t):=J_m(\kappa t)$. By making a change of variables, we easily derive the corresponding results for $J_{m,\kappa}(t)$ which shall be used in the later deduction.
\begin{eqnarray}
  J_{m,\kappa}'(t) &=& -\frac{m}{t}J_{m,\kappa}(t)+\kappa J_{m-1,\kappa}(t) \label{n6e9}\\
  J_{m,\kappa}'(t) &=& -\kappa J_{m+1,\kappa}(t)+\frac{m}{t}J_{m,\kappa}(t) \label{n6e10}\\
  \frac{d}{dt}\left[t^mJ_{m,\kappa}(t)\right] &=& \kappa t^mJ_{m-1,\kappa}(t) \label{n6e5}\\
  \frac{d}{dt}\left[t^{-m}J_{m,\kappa}(t)\right] &=& -\kappa t^{-m}J_{m+1,\kappa}(t) \label{n6e6}
\end{eqnarray}
Specially, we have directly from (\ref{n6e5}) and (\ref{n6e6}) that
\begin{equation}
I_1(m+1,m,\kappa,t)=\frac{1}{\kappa}t^{m+1}J_{m+1,\kappa}(t)
\end{equation}
and
\begin{equation}
I_1(0,1,\kappa,t)=-\frac{1}{\kappa}J_{0,\kappa}(t).
\end{equation}

We next derive some basic recursive relations for $I_1(n,m,\kappa,t)$.
\begin{lemma} For $n,m\in \aZ$ and $|\kappa|>0$
\begin{eqnarray}
  I_{1}(n,m,\kappa,t) &=& \frac{1}{n-m+1}t^{n+1}J_{m,\kappa}(t)-\frac{\kappa}{n-m+1}I_{1}(n+1,m-1,\kappa,t),\label{n6e8} \\
  I_{1}(n,m,\kappa,t) &=& \frac{1}{n+m+1}t^{n+1}J_{m,\kappa}(t)+\frac{\kappa}{n+m+1}I_{1}(n+1,m+1,\kappa,t), \label{n6e7}\\
  I_{1}(n,m,\kappa,t) &=& \frac{1}{\kappa}t^nJ_{m+1,\kappa}(t)-\frac{n-m-1}{\kappa}I_{1}(n-1,m+1,\kappa,t), \label{n6e12}\\
  I_{1}(n,m,\kappa,t) &=& \frac{-1}{\kappa}t^{n}J_{m-1,\kappa}(t)+\frac{n+m-1}{\kappa}I_{1}(n-1,m-1,\kappa,t), \label{n6e11}
\end{eqnarray}
where $n\neq m-1$ in (\ref{n6e8}) and $n\neq -m-1$ in (\ref{n6e7}).
\end{lemma}
\begin{proof}
 Since
 \[I_{1}(n,m,\kappa,t)=\int  J_{m,\kappa}(t)d\frac{t^{n+1}}{n+1}=\frac{t^{n+1}J_{m,\kappa}(t)}{n+1}-\frac{1}{n+1}\int t^{n+1} J_{m,\kappa}'(t)dt,\]
 we obtain (\ref{n6e8}) and (\ref{n6e7}) by substituting $J_{m,\kappa}'$ with the formula (\ref{n6e9}) and (\ref{n6e10}), respectively.

 According to (\ref{n6e5}) and (\ref{n6e6}), we get that
 \begin{eqnarray*}
    I_1(n,m,\kappa,t) &=& \frac{1}{\kappa}\int t^{n-m-1}d\left(t^{m+1}J_{m+1,\kappa}(t)\right) \\
    I_1(n,m,\kappa,t) &=& -\frac{1}{\kappa}\int t^{n+m-1}d\left(t^{-(m-1)}J_{m-1,\kappa}(t)\right).
 \end{eqnarray*}
 Then the formulas (\ref{n6e12}) and (\ref{n6e11}) can be obtained by the integration by parts from the above two equations, respectively.
\end{proof}

Note that the formula (\ref{n6e12}) has also been derived in \citep{XIANG2013} pp. 249.
%While using the recursive relations, we may avoid using (\ref{n6e8}) and (\ref{n6e7}) alone when $\kappa$ is very large.
We combine some of the recursive relations together to get other recursive relations which are helpful in latter deduction. They are given as follows.
\begin{eqnarray}
%  I_{1}(n,m,\kappa,t) &=&  \frac{t^{n+1}J_{m,\kappa}(t)}{n-m+1}-\frac{\kappa t^{n+2}J_{m-1,\kappa}(t)}{(n+1)^2-m^2} -\frac{\kappa^2}{(n+1)^2-m^2} I_{1}(n+2,m,\kappa,t) \label{n6e13} \\
  I_{1}(n,m,\kappa,t) &=& \frac{t^{n+1}(J_{m,\kappa}(t)+J_{m-2,\kappa}(t))}{n-m+1}-\frac{n+m-1}{n-m+1}I_{1}(n,m-2,\kappa,t) \label{n6e14} \\
  I_{1}(n,m,\kappa,t) &=& \frac{2(m+2)t^nJ_{m+1,\kappa}(t)}{\kappa}- \frac{n-m-1}{n+m+1}I_{1}(n,m+2,\kappa,t) \label{n6e15} \\
  I_{1}(n,m,\kappa,t) &=& \frac{t^nJ_{m+1,\kappa}(t)}{\kappa}+\frac{(n-m-1) t^{n-1}J_{m,\kappa}(t)}{\kappa^2} -\frac{(n-1)^2-m^2}{\kappa^2}I_{1}(n-2,m,\kappa,t)\ \label{n6e16}
\end{eqnarray}
We note that
%(\ref{n6e13}) is obtained by combining (\ref{n6e8}) and (\ref{n6e7}),
is obtained by combining (\ref{n6e14}) by (\ref{n6e8}) and (\ref{n6e11}),
(\ref{n6e15}) by (\ref{n6e7}) and (\ref{n6e12}), and (\ref{n6e16}) by (\ref{n6e12}) and (\ref{n6e11}).
%It is obvious that (\ref{n6e14}-\ref{n6e16}) are stable recursive relations except (\ref{n6e13}) when $\kappa$ is very large. Thus we shall avoid the use of (\ref{n6e13}).
We also point out that the formula (\ref{n6e22}) can be obtained easily by the iteration of equation (\ref{n6e15}).

In the next, we present the explicit expressions for $I_{1}(n,m,\kappa,t)$ with different $n$ and $m$ through the preceding recursive relations. Since the expression are obtained by iteration and can be proved by induction easily, we omit the detailed proof. Let $\aN_0:=\aN\cup \{0\}$.

When $n=m+1+2s$ where $s\in \aN_0$, $I_{1}(n,m,\kappa,t)$ has the closed form and can be obtained through (\ref{n6e16}).
\begin{proposition} For $m, s\in\aN_0$ and $|\kappa|>0$,
\begin{equation}\label{n6e31}
  I_{1}(m+1+2s,m,\kappa,t) = \frac{1}{\kappa}t^{m+1}J_{m+1,\kappa}(t)\sum_{j=0}^s c_{s-j}t^{2j}+ \frac{1}{\kappa^2}t^{m}J_{m,\kappa}(t)\sum_{j=1}^s c_{s-j}2jt^{2j}
\end{equation}
where $c_0=1$ and
\[
c_j=\prod_{k=0}^{j-1}\left(-\frac{4(s-k)(m+s-k)}{\kappa^2}\right),\ j>0.
\]
\end{proposition}
We note that $c_j$ are generic constants and their expressions may change in each appearance.

When $n=m$, $I_{1}(n,m,\kappa,t)$ has no closed form but can be simplified through (\ref{n6e11}) to the case of $I_{1}(0,0,\kappa,t)$ which can be evaluated efficiently.
\begin{proposition} For $m\in\aN_0$ and $|\kappa|>0$,
\begin{equation}
  I_{1}(m,m,\kappa,t) = -\frac{1}{\kappa}\sum_{j=1}^{m}c_{m-j}t^{j}J_{j-1,\kappa}(t)+c_mI_{1}(0,0,\kappa,t)
\end{equation}
where $c_0=1$ and
\[
c_j=\prod_{k=0}^{j-1}\frac{2(m-k)-1}{\kappa},\ j>0.
\]
\end{proposition}

When $n=m+2s$ where $s\in \aN_0$, $I_{1}(n,m,\kappa,t)$ can be transformed into the case of $n=m$ with the help of (\ref{n6e16}).
\begin{proposition}
For $m, s\in\aN_0$ and $|\kappa|>0$,
\begin{equation}\label{n6e30}
  I_{1}(m+2s,m,\kappa,t) = \frac{1}{\kappa}t^{m}J_{m+1,\kappa}(t)\sum_{j=1}^{s} c_{s-j}t^{2j}+ \frac{1}{\kappa^2}t^{m-1}J_{m,\kappa}(t)\sum_{j=1}^s c_{s-j}(2j-1)t^{2j}+c_sI_{1}(m,m,\kappa,t)
\end{equation}
where $c_0=1$ and
\[
c_j=\prod_{k=0}^{j-1}\left(-\frac{[2(s-k)-1][2(m+s-k)-1]}{\kappa^2}\right),\ j>0.
\]
\end{proposition}

For the case of $m>n$, we present the corresponding stable recursive formulas with the help of (\ref{n6e11}) and (\ref{n6e14}). By iteration, we get from (\ref{n6e11}) that
\begin{equation}\label{n6e27}
  I(n,m,\kappa,t)=-\sum_{j=0}^{n-1}\frac{c_j}{\kappa}t^{n-j}J_{m-j-1,\kappa}(t)+c_nI_1(0,m-n,\kappa,t)
\end{equation}
where $c_0=1$ and
\[
c_j=\prod_{l=0}^{j-1}\frac{n+m-2l-1}{\kappa},\ j>0.
\]

Since the case of $I_1(0,1,\kappa,t)$ and $I_1(0,0,\kappa,t)$ has been solved, $I_1(0,m-n,\kappa,t)$ can be derived iteratively through (\ref{n6e14}). We then present the last two formulas for the purpose of calculation.

\begin{proposition} \label{n6p2}
  For $n, s\in\aN_0$ and $|\kappa|>0$,
  \begin{equation}\label{n6e28}
    \begin{split}
    I_{1}(n,n+1+2s,\kappa,t) =-&\sum_{j=2}^n\frac{c_{n-j}}{\kappa}t^jJ_{j+2s,\kappa}(t)- \left(\frac{c_{n-1}}{\kappa}+\frac{c_{n}}{2s} \right)t J_{2s+1,\kappa}(t)  \\ & -c_n\sum_{j=1}^{s-1}\frac{2j+1}{2j(j+1)}tJ_{2j+1,\kappa}(t)-\frac{c_n}{2}tJ_{1,\kappa}(t)-\frac{c_n}{\kappa} J_{0,\kappa}(t)
    \end{split}
  \end{equation}
  where $c_0=1$ and
  \[
  c_j=\prod_{l=0}^{j-1}\frac{2(n+s-l)}{\kappa}, \ j>0,
  \]
  and
  \begin{equation}\label{n6e29}
    \begin{split}
    I_{1}(n,n+2s,\kappa,t) =-&\sum_{j=2}^n\frac{c_{n-j}}{\kappa}t^jJ_{j+2s-1,\kappa}(t)- \left(\frac{c_{n-1}}{\kappa}+\frac{c_{n}}{2s-1} \right) t J_{2s,\kappa}(t)  \\ & -c_n\sum_{j=1}^{s-1}\frac{4j}{4j^2-1}tJ_{2j,\kappa}(t)- c_n tJ_{0,\kappa}(t)+c_nI(0,0,\kappa,t)
    \end{split}
  \end{equation}
  where $c_0=1$ and
  \[
  c_j=\prod_{l=0}^{j-1}\frac{2(n+s-l)-1}{\kappa}, \ j>0.
  \]
\end{proposition}
Note that when $s$ or $n$ equals 0 in Proposition \ref{n6p2}, the formula forms need make the obvious adjustment which is omitted here.

Formulas (\ref{n6e31}), (\ref{n6e30}), (\ref{n6e28}) and (\ref{n6e29}) form a complete scheme for the evaluation of $I_1(n,m,\kappa,b)$. With a little calculation, we can find that the absolute value of $c_j$ appeared in (\ref{n6e31}), (\ref{n6e30}), (\ref{n6e28}) and (\ref{n6e29}) is no more than $1$ when $|\kappa|\geq \max(n,m)$. Therefore, based on the recursive formulas, the scheme is especially fast and efficient when $|\kappa|\geq \max(n,m)$ and it is still applicable for the case when $\kappa b<0$. To better evaluate the moments of $I_1(n,m,\kappa,b)$, it is better to combine these formulas and the formulas (\ref{n6e22}) and (\ref{n6e23}) together since (\ref{n6e22}) is efficient for the case of $|\kappa|<\max(n,m)$  while (\ref{n6e23}) is efficient for very large $\kappa$. It will be illustrated in the section of Numerical results and a suggested application domain for each formula will be presented, too.

\section{Evaluation of $I_2(n,m,\kappa,b)$}
We study the evaluation of $I_2(n,m,\kappa,b)$ in this section. According to our analysis, they has the closed form. Therefore, we only need to analyze the corresponding indefinite integrals, denoted by $I_2(n,m,\kappa,t):=\int t^n J_{m,\kappa}(t) e^{i\kappa t}dt$.

We first present two main recursive relations for $I_2(n,m,\kappa,t)$ by using the technique of integration by parts and the property of Bessel functions.
\begin{lemma}For any integer $n$ and $m$, there exists
\begin{eqnarray}
  I_{2}(n,m,\kappa,t) &=& \frac{e^{i\kappa t}t^{n+1}}{n+m+1}\left[J_{m,\kappa}(t)-iJ_{m+1,\kappa}(t)\right]+ \frac{i(n-m)}{n+m+1} I_{2}(n,m+1,\kappa,t),\ \quad \quad\ \label{n6e17}\\
  I_{2}(n,m,\kappa,t) &=& \frac{e^{i\kappa t}t^{n+1}}{n-m+1}\left[J_{m,\kappa}(t)+iJ_{m-1,\kappa}(t)\right]-\frac{i(n+m)}{n-m+1} I_{2}(n,m-1,\kappa,t), \ \quad \quad\ \label{n6e18}
\end{eqnarray}
where $n+m+1\neq 0$ in (\ref{n6e17}) and $n-m+1\neq 0$ in (\ref{n6e18}).

Specially, if $n=m$ in (\ref{n6e17}), then
\begin{equation}\label{n6e19}
  I_{2}(n,n,\kappa,t) =\frac{e^{i\kappa t}t^{n+1}}{2n+1}\left[J_{n,\kappa}(t)-iJ_{n+1,\kappa}(t)\right].
\end{equation}
\end{lemma}
\begin{proof}
  We first prove the formula (\ref{n6e17}). By using the integration by part, we easily get that
  \begin{equation*}
  \begin{split}
    I_{2}(n,m,\kappa,t)& =\frac{1}{n+1}\int J_{m,\kappa}(t)e^{i\kappa t}dt^{n+1} \\
    &=\frac{e^{i\kappa t}t^{n+1}}{n+1}J_{m,\kappa}(t)-\frac{1}{n+1}\int t^{n+1}e^{i\kappa t} \left(i\kappa J_{m,\kappa}(t) +J_{m,\kappa}'(t)\right)dt
  \end{split}
  \end{equation*}
  With the formula (\ref{n6e10}) for $J_{m,\kappa}'$, the above equation can be rewritten as
  \begin{equation*}
    I_{2}(n,m,\kappa,t)=\frac{e^{i\kappa t}t^{n+1}}{n+m+1} J_{m,\kappa}(t)-\frac{\kappa}{n+m+1}\left[iI_{2}(n+1,m,\kappa,t) -I_{2}(n+1,m+1,\kappa,t) \right].
  \end{equation*}
  We next apply equation (\ref{n6e5}) to $I_{2}(n+1,m,\kappa,t)$ and then use again the integration by part which shall give us that
  \begin{equation*}
  \begin{split}
    i\kappa I_{2}(n+1,m,\kappa,t)& =i \int t^{n-m}e^{i\kappa t}d\left( t^{m+1}J_{m+1,\kappa}(t)\right) \\
    &= it^{n+1}e^{i\kappa t}J_{m+1,\kappa}(t)-i(n-m)I_{2}(n,m+1,\kappa,t) +\kappa I_{2}(n+1,m+1,\kappa,t)
  \end{split}
  \end{equation*}
  Combining the above two equations, we obtain the desired formula (\ref{n6e17}).

  For formula (\ref{n6e18}), the proof is similar as that of formula (\ref{n6e17}). We first obtain by using the integration by part directly that
  \begin{equation*}
    I_{2}(n,m,\kappa,t)=\frac{e^{i\kappa t}t^{n+1}}{n-m+1}J_{m,\kappa}(t)-\frac{\kappa}{n-m+1}\left[i I_{2}(n+1,m,\kappa,t)+ I_{2}(n+1,m-1,\kappa,t)\right]
  \end{equation*}
  With the help of formula (\ref{n6e6}) and by using again the technique of integration by part, the integral $i\kappa I_{2}(n+1,m,\kappa,t)$ has the following expression,
  \begin{equation*}
  \begin{split}
    i\kappa I_{2}(n+1,m,\kappa,t)& =-i\int t^{n+m}e^{i\kappa t}d\left(t^{1-m}J_{m-1,\kappa}(t)\right) \\
    &=-it^{n+1}e^{i\kappa t}J_{m-1,\kappa}(t)+i(n+m)I_{2}(n,m-1,\kappa,t)-\kappa I_{2}(n+1,m-1,\kappa,t)
  \end{split}
  \end{equation*}
   Substituting the expression of $i\kappa I_{2}(n+1,m,\kappa,t)$ into $I_{2}(n,m,\kappa,t)$, we get the desired formula (\ref{n6e18}).
\end{proof}

With the formulas (\ref{n6e17}) and (\ref{n6e19}), the integrals $I_2(n,m,\kappa,t)$ with $n\geq m$ are easily obtained by iteration. We present the expression of $I_2(n,m,\kappa,t)$ without a proof.
\begin{proposition}
  For $n,m\in\aN_0$, $n\geq m$ and $|\kappa|>0$,
  \begin{equation}\label{n6e32}
    I_2(n,m,\kappa,t)=e^{i\kappa t}t^{n+1}\sum_{j=m}^n\frac{c_{j-m}}{n+j+1}\left(J_{j,\kappa}(t) -iJ_{j+1,\kappa}(t) \right)
  \end{equation}
  where $c_0=1$ and
  \[
  c_j=\prod_{k=m}^{m+j-1}\frac{i(n-k)}{n+k+1}, \ j>0.
  \]
\end{proposition}

However, it fails to evaluate the case $I_2(n,m,\kappa,t)$ with $n<m$ by the formulas (\ref{n6e18}) and (\ref{n6e19}). It is because that $(\ref{n6e18})$ fails for $n=m-1$. To solve this problem, we next present an explicit expression for the case $I_2(n,n+1,\kappa,t)$.
\begin{proposition}\label{n6p1}
For $n\in\aN_0$ and $|\kappa|>0$,
  \begin{equation}\label{n6e20}
  \begin{split}
    I_2(n,n+1,\kappa,t)& =c_n\left(it-\frac{1}{\kappa}\right)e^{i\kappa t}J_{0,\kappa}(t) +\sum_{j=1}^n\left( \frac{i t c_{n-j}}{2j+1}-\frac{c_{n-j}}{\kappa}+\frac{c_{n-j+1}}{2j-1}\right)t^j e^{i\kappa t}J_{j,\kappa}(t) \\
    &\quad\quad\quad\quad \quad\quad\quad\quad \quad\quad\quad +\frac{c_0}{2n+1} t^{n+1}e^{i\kappa t}J_{n+1,\kappa}(t)
    \end{split}
  \end{equation}
  where $c_0=1$ and
  \[
  c_j=\prod_{k=0}^{j-1}\frac{2(n-k)}{\kappa},\ j>0.
  \]
\end{proposition}
\begin{proof}
    With the help of formula (\ref{n6e6}) and by  the technique of integration by part, we have that
  \begin{equation*}
  \begin{split}
    I_{2}(n,m,\kappa,t)& =-\frac{1}{\kappa}\int t^{n-1+m}e^{i\kappa t}d\left(t^{1-m}J_{m-1,\kappa}(t)\right) \\
    &=-\frac{1}{\kappa} t^{n}e^{i\kappa t}J_{m-1,\kappa}(t)+ \frac{1}{\kappa}(n-1+m)I_{2}(n-1,m-1,\kappa,t)+iI_{2}(n,m-1,\kappa,t)
  \end{split}
  \end{equation*}
   Setting $m=n+1$, we get a recursive formula for $I_2(n,n+1,\kappa,t)$ that
  \begin{equation}
    I_{2}(n,n+1,\kappa,t)=-\frac{e^{i\kappa t}t^n}{\kappa} J_{n,\kappa}(t) + \frac{2n}{\kappa} I_{2}(n-1,n,\kappa,t) +iI_{2}(n,n,\kappa,t), \ n\in\aZ.
  \end{equation}
  Specially, we have for $n=0$ that $I_{2}(0,1,\kappa,t)=-\frac{e^{i\kappa t}}{\kappa} J_{0,\kappa}(t)+iI_{2}(0,0,\kappa,t)$. Hence, the proof is easily finished by induction.
\end{proof}

With Proposition \ref{n6p1}, we can obtain the explicit formula for $I_2(n,m,\kappa,t)$ with $n<m$ by the recursive use of (\ref{n6e18}). We give the corresponding results in the next proposition without a proof.
\begin{proposition}
  For $n, m\in\aN_0$, $n<m$ and $|\kappa|>0$,
  \begin{equation}\label{n6e33}
    I_2(n,m,\kappa,t)=t^{n+1}e^{i\kappa t}\sum_{j=n+2}^m\frac{c_{m-j}}{n-j+1} \left[J_{j,\kappa}(t)+iJ_{j-1,\kappa}(t)\right] +c_{m-n-1}I_2(n,n+1,\kappa,t)
  \end{equation}
  where $c_0=1$ and
  \[
  c_j=\prod_{k=m+1-j}^{m}\frac{-i(n+k)}{n-k+1}, \ j>0.
  \]
\end{proposition}

Formulas (\ref{n6e32}), (\ref{n6e20}) and (\ref{n6e33}) give a practical way to analyze the moments $I_2(n,m,\kappa,b)$ than that of formulas (\ref{n6e24}) and (\ref{n6e25}). We also note that formulas (\ref{n6e20}) and (\ref{n6e33}) are more suitable for the case $|\kappa|\geq 2n$ which is quite common in practice since the error will not be amplified during the iteration.

\section{Numerical Results}
We present several numerical results to validate the accuracy of the formulas proposed in Section 2 for $I_1(n,m,\kappa,b)$ and then compare the computation time with the formulas (\ref{n6e22}) and (\ref{n6e23}) to determine the application range of each method. We shall not present numerical experiments for $I_2(n,m,\kappa,b)$ since there is no other proper methods to compare with and the formulas deduced for $I_2(n,m,\kappa,b)$ are all in the closed form. The computation was done by the software Matlab on a laptop with an Intel(R) Core(TM) i5-4200U CPU @ 1.60GHZ 2.30GHz.

For the accuracy, five pairs of values for $[n,m]$ are selected: $[0,0], [5,3], [5,4], [5,6]$, and $[5,7]$ and three typical values of $\kappa$ are chosen: $1, 10$ and $100$. Let the parameter $b$ range from $0.1$ to $1$ with an interval $0.01$. The reference values for $I_1(n,m,\kappa,b)$ are obtained from the scheme, when $\kappa b\leq60$, using formula (\ref{n6e22}) with 100 truncated terms and when $\kappa b>60$, using formula (\ref{n6e23}) with 20 truncated terms. The accuracy of the reference values have also been validated by Mathematics 8.0. The absolute errors of $I_1(n,m,\kappa,b)$ for these cases calculated by the formulas proposed in Section 2 are shown in Figs. \ref{n6f3}-\ref{n6f7}. Some of the error curves in these figures are broken and the reason is that we plot the errors with logarithmic scale while some errors computed by the software Matlab are zero.  Fig. \ref{n6f3} validates the accuracy within the machine tolerance of the scheme proposed in evaluation of $I_1(0,0,\kappa,b)$. Since the evaluations of $I_1(5,3,\kappa,b)$ and $I_1(5,7,\kappa,b)$ depends on $I_1(0,0,\kappa,b)$ by formulas (\ref{n6e30}) and (\ref{n6e29}), respectively, the error of $I_1(0,0,\kappa,b)$ will transfer largely to $I_1(5,3,\kappa,b)$ and $I_1(5,7,\kappa,b)$ when $\kappa$ is relatively small with respect to $n$ and $m$. It is why the errors increase when $\kappa=1$ in Figs. \ref{n6f4} and \ref{n6f7}. According to formulas (\ref{n6e31}) and (\ref{n6e28}), $I_1(5,4,\kappa,b)$ and $I_1(5,6,\kappa,b)$ should give the exact values. In fact,  however, $I_1(5,6,\kappa,b)$ has relatively large errors  when $\kappa=1$ shown in Fig. \ref{n6f6}. It is because that the evaluations of Bessel functions have small errors and they can be amplified by the iteration when $\kappa$ is small and then transfer to $I_1(5,6,\kappa,b)$. Among all the figures, we can derive that the formulas derived by iteration in Section 2 behaves greatly when $\kappa$ is relatively large with respect to $n$ and $m$ and easily reach the machine tolerance.

%\begin{figure}
\begin{minipage}[t]{0.5\linewidth}
\centering
\includegraphics[width=3in]{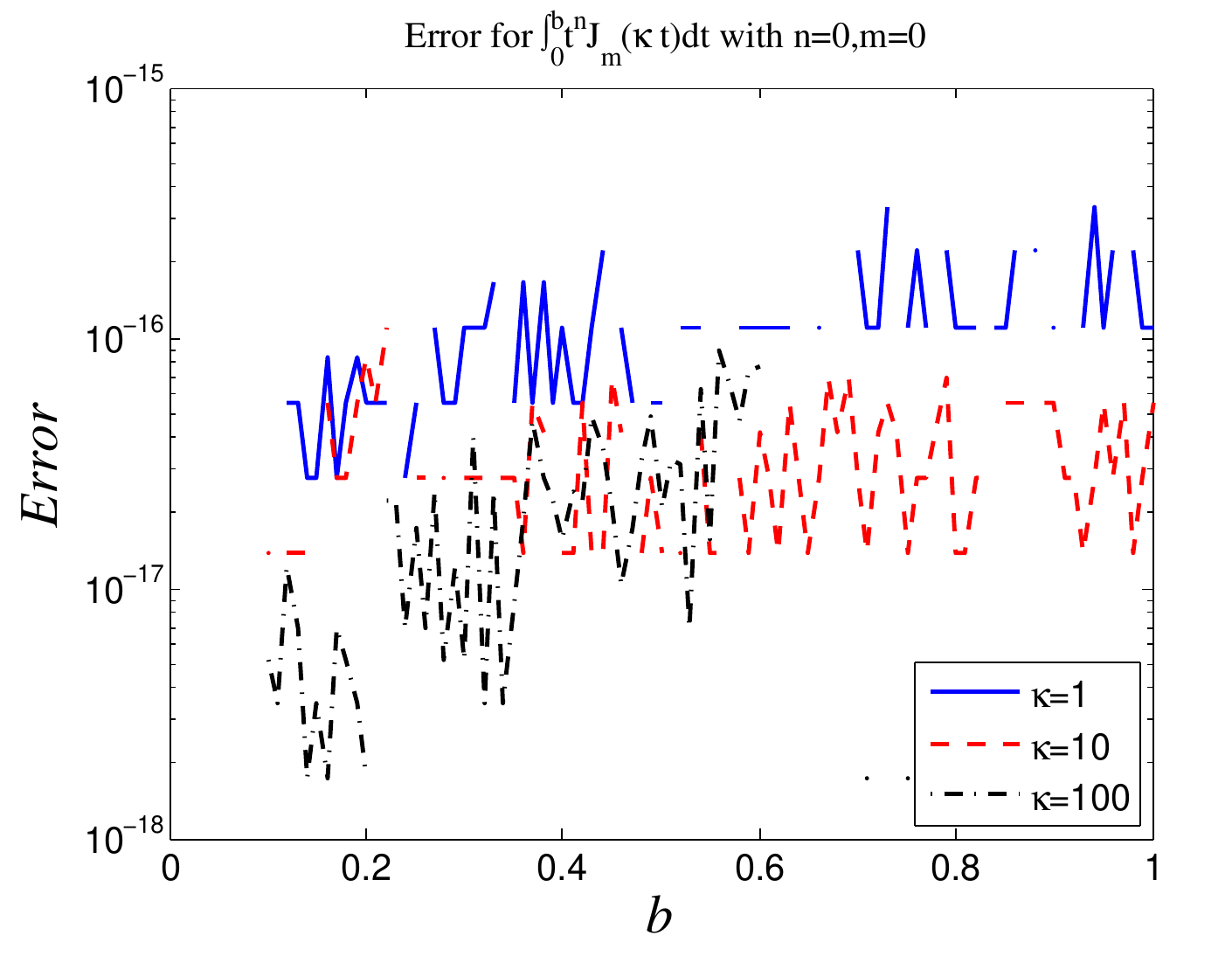}
\figcaption{The absolute errors of $I_1(0,0,\kappa,b)$}
\label{n6f3}
\end{minipage}%
\begin{minipage}[t]{0.5\linewidth}
\centering
\includegraphics[width=3in]{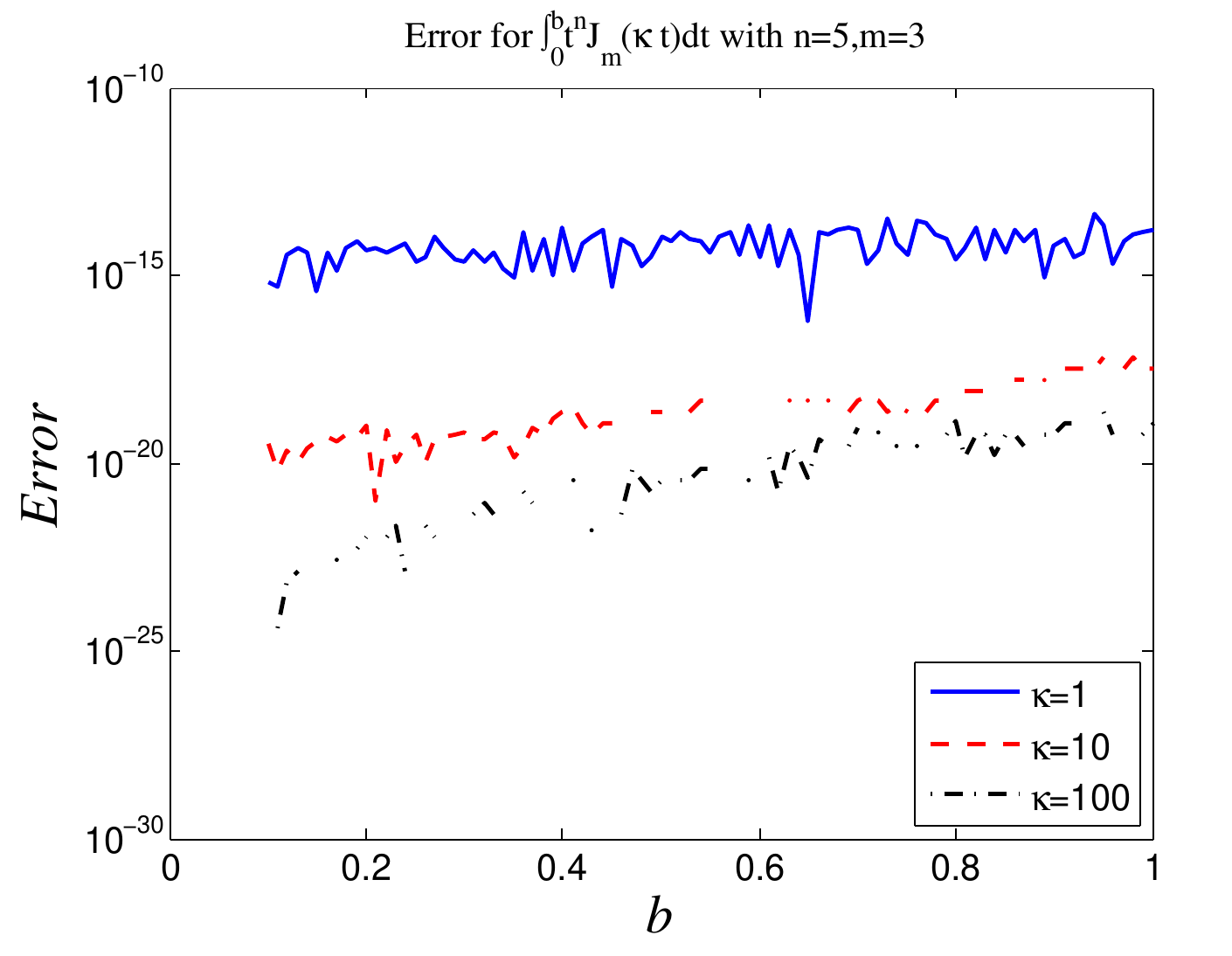}
\figcaption{The absolute errors of $I_1(5,3,\kappa,b)$}
\label{n6f4}
\end{minipage}
%\end{figure}

%\begin{figure}
\begin{minipage}[t]{0.5\linewidth}
\centering
\includegraphics[width=3in]{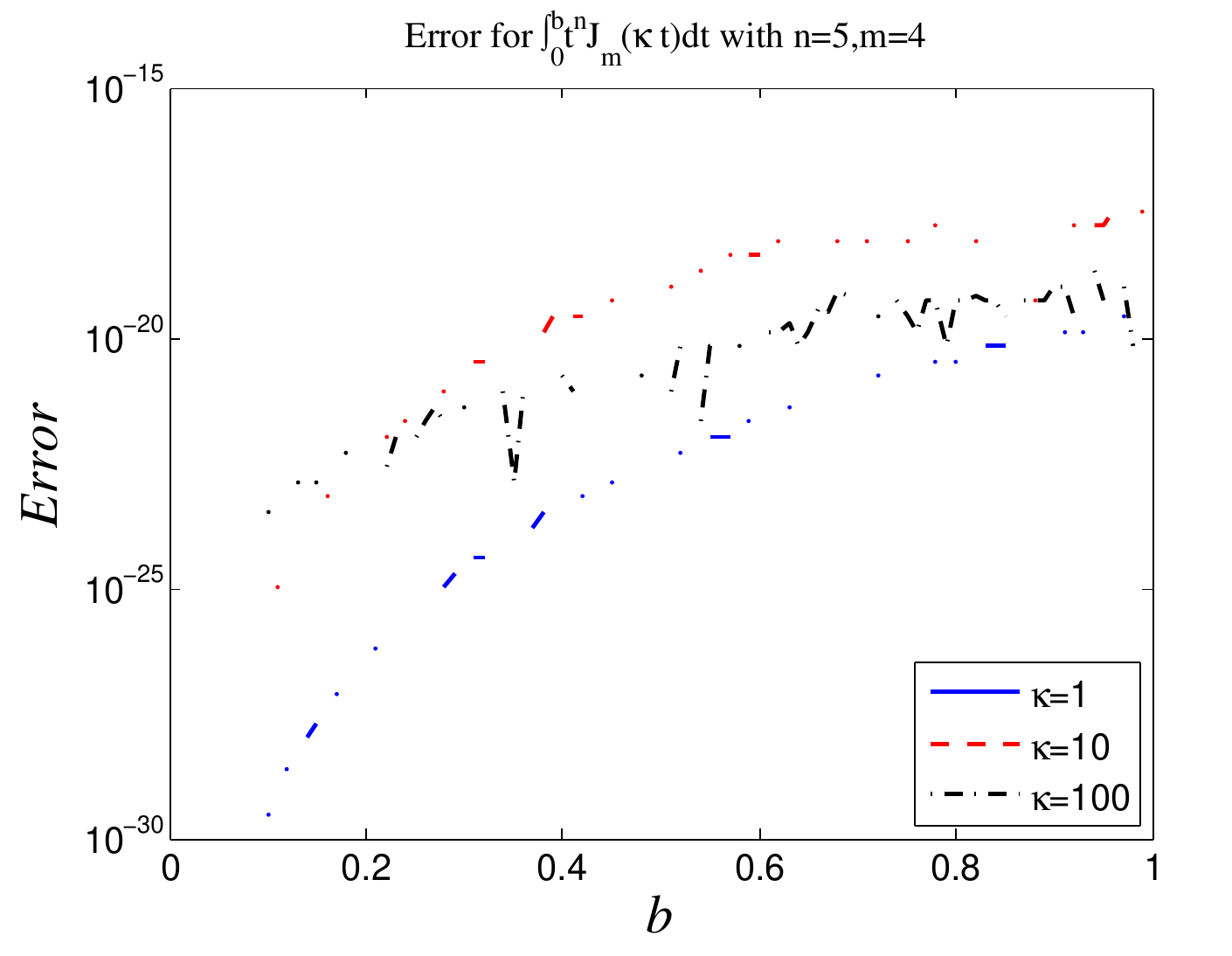}
\figcaption{The absolute errors of $I_1(5,4,\kappa,b)$}
\label{n6f5}
\end{minipage}%
\begin{minipage}[t]{0.5\linewidth}
\centering
\includegraphics[width=3in]{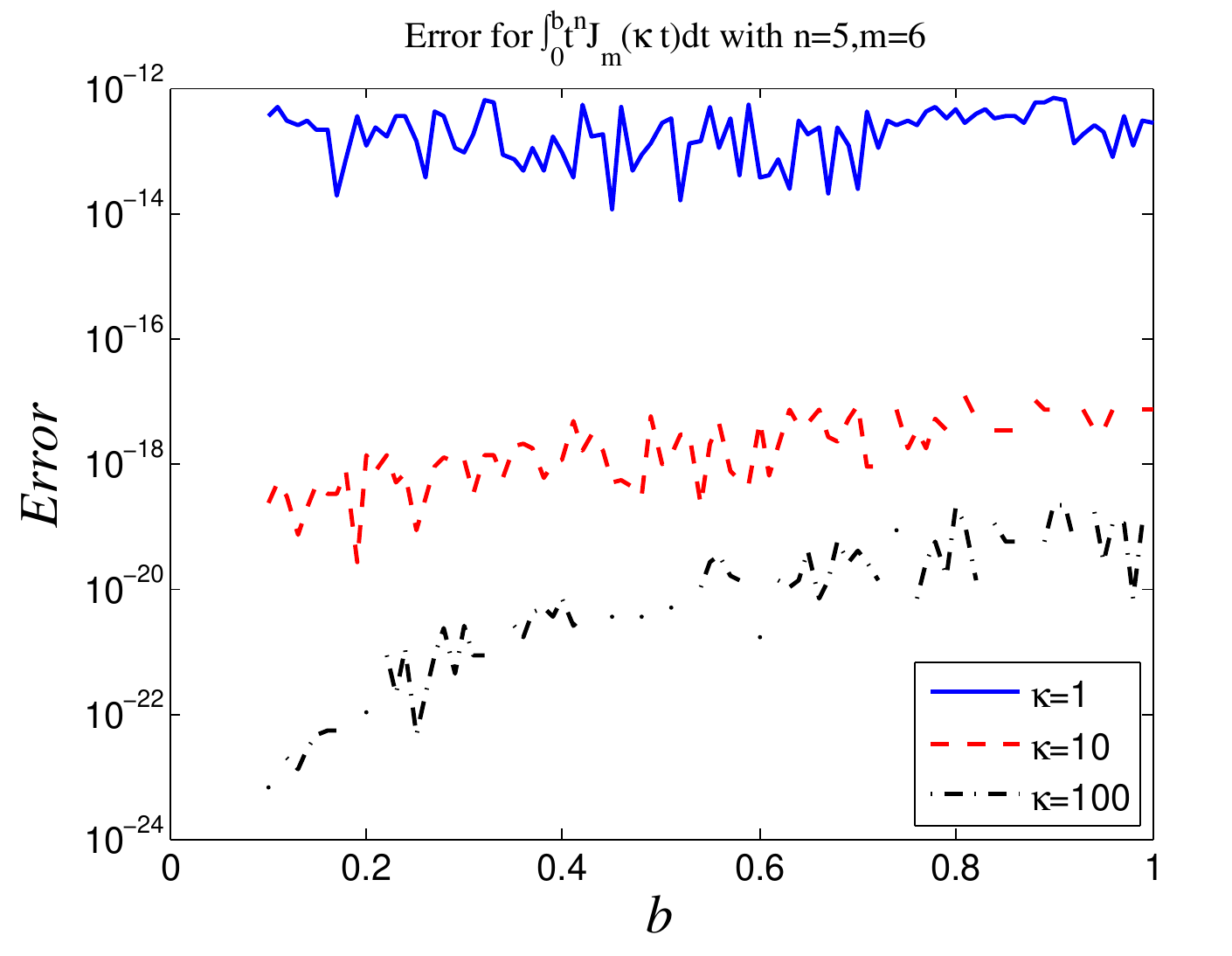}
\figcaption{The absolute errors of $I_1(5,6,\kappa,b)$}
\label{n6f6}
\end{minipage}
%\end{figure}

%\begin{figure}
\begin{minipage}[t]{0.5\linewidth}
\centering
\includegraphics[width=3in]{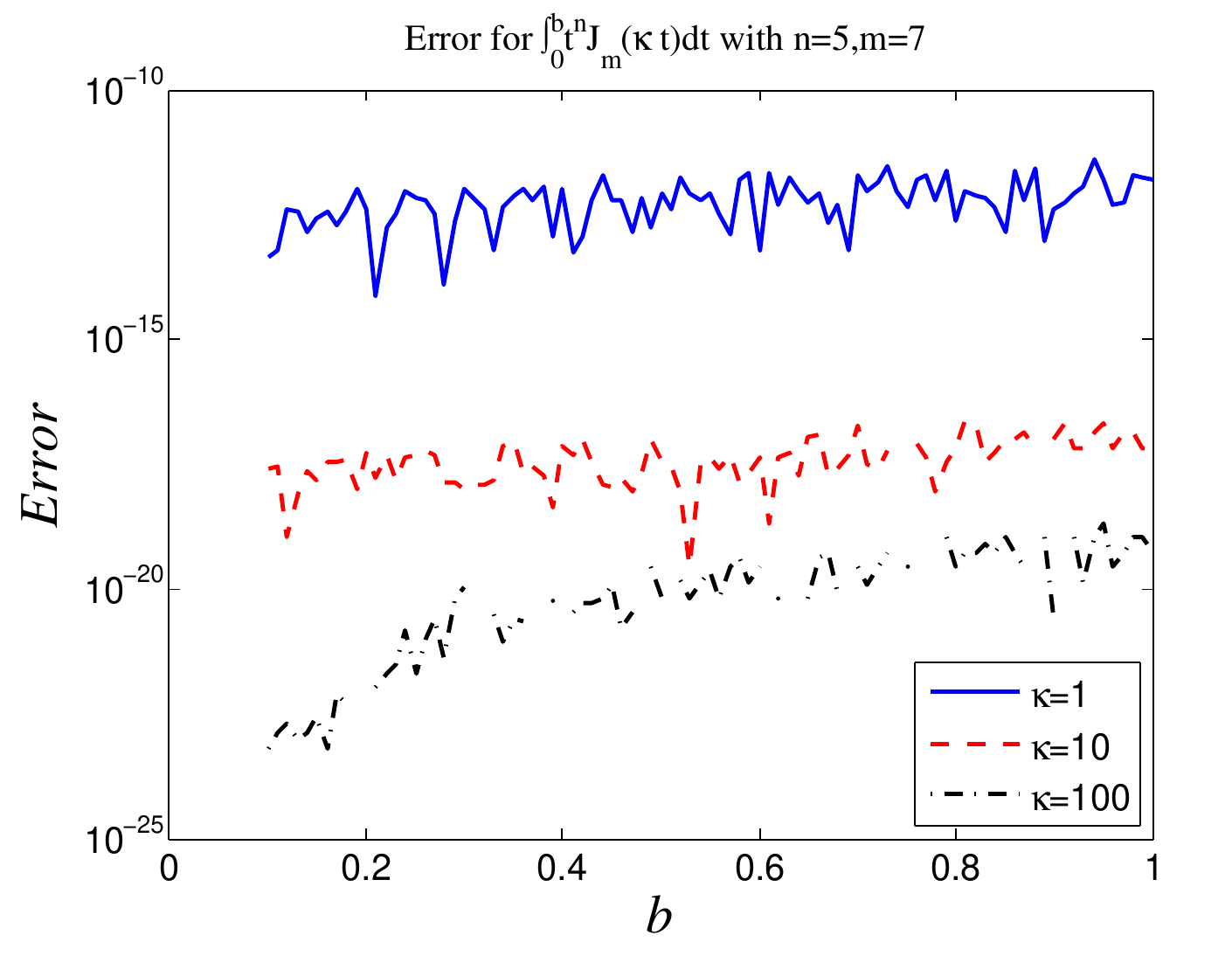}
\figcaption{The absolute errors of $I_1(5,7,\kappa,b)$}
\label{n6f7}
\end{minipage}%
\begin{minipage}[t]{0.5\linewidth}
\centering
\includegraphics[width=2in]{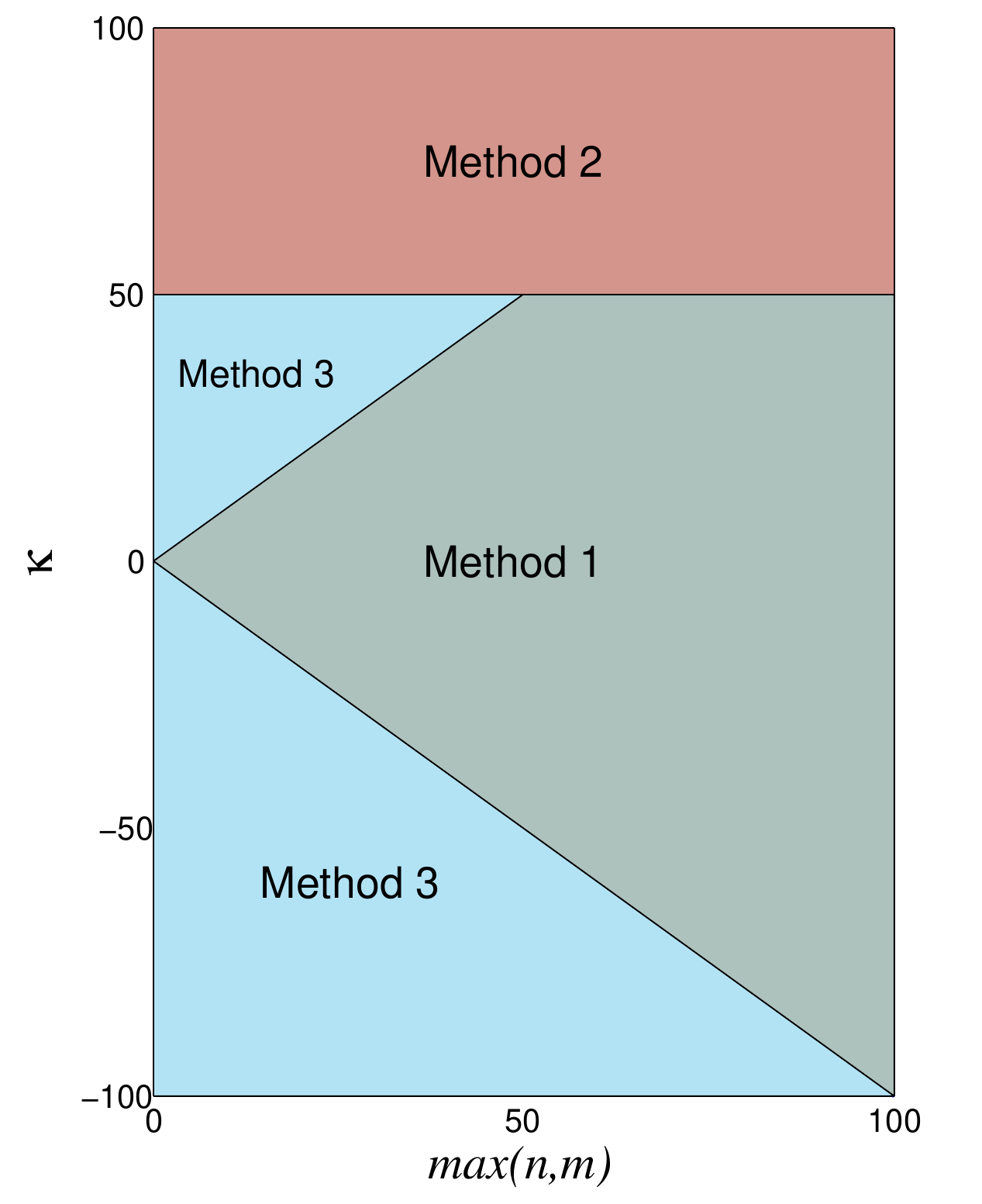}
\figcaption{A suggested application domain for each method.}
\label{n6f8}
\end{minipage}
%\end{figure}

We next carry out an numerical experiment by comparing the efficiency of each formula in evaluating $I_1(n,m,\kappa,b)$. For this purpose, we denote the formula (\ref{n6e22}) with certain truncated terms method 1, the formula (\ref{n6e23}) with certain truncated terms method 2 and the formulas derived in Section 2 method 3. Let $b$ be fixed 1, $n, m$ range in $[0,16]$  and $\kappa$ in $[1, 100]$.  The reference values for $I_1(n,m,\kappa,b)$ are derived numerically as the first numerical experiment. We record the computation time of each method running 100 times when the error for each method with proper truncated terms reaches the machine tolerance. For method 1, the number of truncated terms increases one by one before it reaches the machine tolerance or the number is bigger than 100. For method 2, the number of truncated terms adds one by one before it reaches the machine tolerance or the number is bigger than 60. The third method does not require such a number. If the methods fail to get the machine tolerance, its corresponding time will be denoted as $Inf$. To have a better view of the tables of computation time for each method, we color the columns of $\kappa=50, 100$ red and the elements of $\kappa\geq\max(n,m)$ blue.
We have a clear observation from Tables \ref{n6ta1}-\ref{n6ta9} that Method 3 is the most efficient when $n+m\leq 1$ while Method 2 is the most efficient when $\kappa \geq 50$. When $\kappa<\max(n,m)$, Method 3 may not touch the machine tolerance which is shown in Table \ref{n6ta3} and \ref{n6ta6} because the iterations happened in this case may not stable. Fortunately, Method 1 performs well when $\kappa<\max(n,m)$ which is illustrated in Tables \ref{n6ta1} and \ref{n6ta4}.
For the case of $n=0$, Method 3 is effective shown in Tables \ref{n6ta7} which seems that it is not influenced by the iterations when $\kappa$ is small. It is because that the Bessel functions decays exponentially when $m>\kappa$.
When $\kappa>\max(n,m)$ and $\kappa<50$, both Method 1 and 3 can reach the machine tolerance and Method 3 is more efficient than Method 1.

In the end, we may present a suggested scheme by combining these three methods to evaluate the moments $I_1(n,m,\kappa,b)$ accurately and fast. The suggested application domain of each method is shown in Fig. \ref{n6f8} and the scheme is: when $\kappa b>50$, method 2 with a proper number (for example, 11,) of truncated terms is used; when $\kappa b<50$ and $|\kappa|\geq \max(n,m)$, Method 3 is adopted; when $\kappa b<50$  and $|\kappa|<\max(n,m)$, method 1 with a proper number (for example, 15,) of truncated terms is the best choice.

%\begin{table}[htb]
\begin{center}
\small
\tabcaption{Computation time of Method 1 with $n=m$ }
\label{n6ta1}
\begin{tabular}{c|c|cccccc}
\hline
n & m & $\kappa=1$ & $\kappa=5$ & $\kappa=10$ & $\kappa=20$ & $\kappa=50$ & $\kappa=100$\\
\hline
$0$&$0$&$\color{blue}{3.93e-2}$&$\color{blue}{7.06e-2}$&$\color{blue}{8.71e-2}$&$\color{blue}{1.29e-1}$&$\color{red}{2.86e-1}$&$\color{red}{Inf}$\\
$1$&$1$&$\color{blue}{3.11e-2}$&$\color{blue}{6.89e-2}$&$\color{blue}{9.91e-2}$&$\color{blue}{1.50e-1}$&$\color{red}{2.75e-1}$&$\color{red}{Inf}$\\
$2$&$2$&$3.75e-2$&$\color{blue}{5.99e-2}$&$\color{blue}{1.17e-1}$&$\color{blue}{1.51e-1}$&$\color{red}{2.32e-1}$&$\color{red}{Inf}$\\
$3$&$3$&$3.72e-2$&$\color{blue}{6.66e-2}$&$\color{blue}{8.11e-2}$&$\color{blue}{1.17e-1}$&$\color{red}{2.57e-1}$&$\color{red}{Inf}$\\
$4$&$4$&$3.37e-2$&$\color{blue}{7.40e-2}$&$\color{blue}{1.00e-1}$&$\color{blue}{1.56e-1}$&$\color{red}{2.77e-1}$&$\color{red}{Inf}$\\
$6$&$6$&$2.71e-2$&$5.57e-2$&$\color{blue}{9.30e-2}$&$\color{blue}{1.41e-1}$&$\color{red}{2.70e-1}$&$\color{red}{5.84e-1}$\\
$8$&$8$&$2.34e-2$&$4.74e-2$&$\color{blue}{8.81e-2}$&$\color{blue}{1.26e-1}$&$\color{red}{1.96e-1}$&$\color{red}{4.43e-1}$\\
$10$&$10$&$1.38e-2$&$3.18e-2$&$\color{blue}{5.21e-2}$&$\color{blue}{9.54e-2}$&$\color{red}{2.45e-1}$&$\color{red}{4.19e-1}$\\
$12$&$12$&$1.26e-2$&$2.65e-2$&$4.45e-2$&$\color{blue}{1.34e-1}$&$\color{red}{1.98e-1}$&$\color{red}{5.05e-1}$\\
$14$&$14$&$1.47e-2$&$2.33e-2$&$4.16e-2$&$\color{blue}{7.33e-2}$&$\color{red}{2.14e-1}$&$\color{red}{2.64e-1}$\\
$16$&$16$&$2.76e-2$&$2.33e-2$&$3.92e-2$&$\color{blue}{1.04e-1}$&$\color{red}{1.40e-1}$&$\color{red}{2.96e-1}$\\
\hline
\end{tabular}
\end{center}
%\end{table}

%\begin{table}[htb]
\begin{center}
\small
\tabcaption{Computation time of Method 2 with $n=m$}
\label{n6ta2}
\begin{tabular}{c|c|cccccc}
\hline
n & m & $\kappa=1$ & $\kappa=5$ & $\kappa=10$ & $\kappa=20$ & $\kappa=50$ & $\kappa=100$\\
\hline
$0$&$0$&$\color{blue}{2.39e-2}$&$\color{blue}{Inf}$&$\color{blue}{Inf}$&$\color{blue}{Inf}$&$\color{red}{1.87e-2}$&$\color{red}{1.64e-2}$\\
$1$&$1$&$\color{blue}{Inf}$&$\color{blue}{Inf}$&$\color{blue}{Inf}$&$\color{blue}{Inf}$&$\color{red}{2.23e-2}$&$\color{red}{2.21e-2}$\\
$2$&$2$&$Inf$&$\color{blue}{Inf}$&$\color{blue}{Inf}$&$\color{blue}{Inf}$&$\color{red}{2.06e-2}$&$\color{red}{1.53e-2}$\\
$3$&$3$&$Inf$&$\color{blue}{Inf}$&$\color{blue}{Inf}$&$\color{blue}{Inf}$&$\color{red}{1.52e-2}$&$\color{red}{2.09e-2}$\\
$4$&$4$&$Inf$&$\color{blue}{Inf}$&$\color{blue}{Inf}$&$\color{blue}{Inf}$&$\color{red}{1.44e-2}$&$\color{red}{2.40e-2}$\\
$6$&$6$&$Inf$&$Inf$&$\color{blue}{Inf}$&$\color{blue}{Inf}$&$\color{red}{3.52e-2}$&$\color{red}{2.63e-2}$\\
$8$&$8$&$Inf$&$Inf$&$\color{blue}{Inf}$&$\color{blue}{Inf}$&$\color{red}{1.97e-2}$&$\color{red}{1.75e-2}$\\
$10$&$10$&$Inf$&$Inf$&$\color{blue}{Inf}$&$\color{blue}{Inf}$&$\color{red}{1.64e-2}$&$\color{red}{2.10e-2}$\\
$12$&$12$&$Inf$&$Inf$&$Inf$&$\color{blue}{Inf}$&$\color{red}{3.29e-2}$&$\color{red}{2.20e-2}$\\
$14$&$14$&$Inf$&$Inf$&$Inf$&$\color{blue}{Inf}$&$\color{red}{3.24e-2}$&$\color{red}{1.49e-2}$\\
$16$&$16$&$Inf$&$Inf$&$Inf$&$\color{blue}{Inf}$&$\color{red}{1.61e-2}$&$\color{red}{2.44e-2}$\\
\hline
\end{tabular}
\end{center}
%\end{table}

%\begin{table}[htb]
\begin{center}
\small
\tabcaption{Computation time of Method 3 with $n=m$}
\label{n6ta3}
\begin{tabular}{c|c|cccccc}
\hline
n & m & $\kappa=1$ & $\kappa=5$ & $\kappa=10$ & $\kappa=20$ & $\kappa=50$ & $\kappa=100$\\
\hline
$0$&$0$&$\color{blue}{6.69e-3}$&$\color{blue}{9.10e-3}$&$\color{blue}{8.42e-3}$&$\color{blue}{6.57e-3}$&$\color{red}{9.12e-3}$&$\color{red}{1.10e-2}$\\
$1$&$1$&$\color{blue}{1.65e-2}$&$\color{blue}{2.36e-2}$&$\color{blue}{1.87e-2}$&$\color{blue}{2.01e-2}$&$\color{red}{2.45e-2}$&$\color{red}{2.19e-2}$\\
$2$&$2$&$Inf$&$\color{blue}{2.71e-2}$&$\color{blue}{6.29e-2}$&$\color{blue}{3.32e-2}$&$\color{red}{3.10e-2}$&$\color{red}{3.01e-2}$\\
$3$&$3$&$Inf$&$\color{blue}{3.96e-2}$&$\color{blue}{2.92e-2}$&$\color{blue}{3.13e-2}$&$\color{red}{5.19e-2}$&$\color{red}{4.27e-2}$\\
$4$&$4$&$Inf$&$\color{blue}{4.66e-2}$&$\color{blue}{3.85e-2}$&$\color{blue}{4.57e-2}$&$\color{red}{5.45e-2}$&$\color{red}{5.09e-2}$\\
$6$&$6$&$Inf$&$7.00e-2$&$\color{blue}{7.95e-2}$&$\color{blue}{9.74e-2}$&$\color{red}{7.23e-2}$&$\color{red}{8.06e-2}$\\
$8$&$8$&$Inf$&$Inf$&$\color{blue}{9.74e-2}$&$\color{blue}{8.03e-2}$&$\color{red}{6.11e-2}$&$\color{red}{6.33e-2}$\\
$10$&$10$&$Inf$&$Inf$&$\color{blue}{8.49e-2}$&$\color{blue}{8.89e-2}$&$\color{red}{1.15e-1}$&$\color{red}{7.48e-2}$\\
$12$&$12$&$Inf$&$Inf$&$1.02e-1$&$\color{blue}{1.59e-1}$&$\color{red}{1.07e-1}$&$\color{red}{8.81e-2}$\\
$14$&$14$&$1.02e-1$&$Inf$&$Inf$&$\color{blue}{1.46e-1}$&$\color{red}{1.04e-1}$&$\color{red}{1.73e-1}$\\
$16$&$16$&$Inf$&$Inf$&$Inf$&$\color{blue}{1.33e-1}$&$\color{red}{1.32e-1}$&$\color{red}{1.71e-1}$\\
\hline
\end{tabular}
\end{center}
%\end{table}

%\begin{table}[htb]
\begin{center}
\small
\tabcaption{Computation time of Method 1 with $m=0$}
\label{n6ta4}
\begin{tabular}{c|c|cccccc}
\hline
n & m & $\kappa=1$ & $\kappa=5$ & $\kappa=10$ & $\kappa=20$ & $\kappa=50$ & $\kappa=100$\\
\hline
$0$&$0$&$\color{blue}{4.47e-2}$&$\color{blue}{8.95e-2}$&$\color{blue}{1.29e-1}$&$\color{blue}{1.97e-1}$&$\color{red}{3.02e-1}$&$\color{red}{Inf}$\\
$1$&$0$&$\color{blue}{1.40e-2}$&$\color{blue}{1.73e-2}$&$\color{blue}{1.80e-2}$&$\color{blue}{1.76e-2}$&$\color{red}{1.57e-2}$&$\color{red}{1.24e-2}$\\
$2$&$0$&$4.16e-2$&$\color{blue}{5.92e-2}$&$\color{blue}{1.21e-1}$&$\color{blue}{1.47e-1}$&$\color{red}{3.03e-1}$&$\color{red}{Inf}$\\
$3$&$0$&$1.43e-2$&$\color{blue}{1.66e-2}$&$\color{blue}{1.46e-2}$&$\color{blue}{1.68e-2}$&$\color{red}{1.07e-2}$&$\color{red}{1.73e-2}$\\
$4$&$0$&$4.80e-2$&$\color{blue}{5.84e-2}$&$\color{blue}{9.59e-2}$&$\color{blue}{1.33e-1}$&$\color{red}{2.20e-1}$&$\color{red}{Inf}$\\
$6$&$0$&$3.61e-2$&$5.60e-2$&$\color{blue}{7.81e-2}$&$\color{blue}{1.05e-1}$&$\color{red}{1.88e-1}$&$\color{red}{3.96e-1}$\\
$8$&$0$&$4.18e-2$&$5.31e-2$&$\color{blue}{7.03e-2}$&$\color{blue}{9.56e-2}$&$\color{red}{1.82e-1}$&$\color{red}{4.14e-1}$\\
$10$&$0$&$3.72e-2$&$4.85e-2$&$\color{blue}{6.99e-2}$&$\color{blue}{9.29e-2}$&$\color{red}{1.48e-1}$&$\color{red}{2.69e-1}$\\
$12$&$0$&$4.49e-2$&$6.40e-2$&$7.08e-2$&$\color{blue}{9.17e-2}$&$\color{red}{1.52e-1}$&$\color{red}{1.89e-1}$\\
$14$&$0$&$3.55e-2$&$5.78e-2$&$7.14e-2$&$\color{blue}{8.93e-2}$&$\color{red}{1.45e-1}$&$\color{red}{1.82e-1}$\\
$16$&$0$&$3.59e-2$&$5.24e-2$&$6.02e-2$&$\color{blue}{9.19e-2}$&$\color{red}{1.17e-1}$&$\color{red}{1.28e-1}$\\
\hline
\end{tabular}
\end{center}
%\end{table}

%\begin{table}[htb]
\begin{center}
\small
\tabcaption{Computation time of Method 2 with $m=0$}
\label{n6ta5}
\begin{tabular}{c|c|cccccc}
\hline
n & m & $\kappa=1$ & $\kappa=5$ & $\kappa=10$ & $\kappa=20$ & $\kappa=50$ & $\kappa=100$\\
\hline
$0$&$0$&$\color{blue}{2.13e-2}$&$\color{blue}{Inf}$&$\color{blue}{Inf}$&$\color{blue}{Inf}$&$\color{red}{2.23e-2}$&$\color{red}{1.43e-2}$\\
$1$&$0$&$\color{blue}{1.76e-2}$&$\color{blue}{2.02e-2}$&$\color{blue}{2.44e-2}$&$\color{blue}{2.25e-2}$&$\color{red}{1.77e-2}$&$\color{red}{1.65e-2}$\\
$2$&$0$&$2.20e-2$&$\color{blue}{Inf}$&$\color{blue}{Inf}$&$\color{blue}{Inf}$&$\color{red}{3.18e-2}$&$\color{red}{1.56e-2}$\\
$3$&$0$&$1.38e-2$&$\color{blue}{1.87e-2}$&$\color{blue}{1.95e-2}$&$\color{blue}{2.06e-2}$&$\color{red}{1.46e-2}$&$\color{red}{1.73e-2}$\\
$4$&$0$&$2.26e-2$&$\color{blue}{Inf}$&$\color{blue}{Inf}$&$\color{blue}{Inf}$&$\color{red}{1.35e-2}$&$\color{red}{1.71e-2}$\\
$6$&$0$&$2.11e-2$&$Inf$&$\color{blue}{Inf}$&$\color{blue}{1.45e-2}$&$\color{red}{1.89e-2}$&$\color{red}{1.84e-2}$\\
$8$&$0$&$2.10e-2$&$Inf$&$\color{blue}{Inf}$&$\color{blue}{1.97e-2}$&$\color{red}{1.75e-2}$&$\color{red}{1.76e-2}$\\
$10$&$0$&$1.75e-2$&$Inf$&$\color{blue}{Inf}$&$\color{blue}{2.14e-2}$&$\color{red}{1.71e-2}$&$\color{red}{1.92e-2}$\\
$12$&$0$&$1.71e-2$&$Inf$&$Inf$&$\color{blue}{1.49e-2}$&$\color{red}{1.46e-2}$&$\color{red}{1.82e-2}$\\
$14$&$0$&$1.98e-2$&$Inf$&$Inf$&$\color{blue}{2.27e-2}$&$\color{red}{1.45e-2}$&$\color{red}{1.67e-2}$\\
$16$&$0$&$1.95e-2$&$Inf$&$Inf$&$\color{blue}{2.03e-2}$&$\color{red}{1.95e-2}$&$\color{red}{1.84e-2}$\\
\hline
\end{tabular}
\end{center}
%\end{table}

%\begin{table}[htb]
\begin{center}
\small
\tabcaption{Computation time of Method 3 with $m=0$}
\label{n6ta6}
\begin{tabular}{c|c|cccccc}
\hline
n & m & $\kappa=1$ & $\kappa=5$ & $\kappa=10$ & $\kappa=20$ & $\kappa=50$ & $\kappa=100$\\
\hline
$0$&$0$&$\color{blue}{7.32e-3}$&$\color{blue}{1.00e-2}$&$\color{blue}{1.08e-2}$&$\color{blue}{1.07e-2}$&$\color{red}{9.85e-3}$&$\color{red}{1.16e-2}$\\
$1$&$0$&$\color{blue}{1.17e-2}$&$\color{blue}{1.39e-2}$&$\color{blue}{1.49e-2}$&$\color{blue}{1.32e-2}$&$\color{red}{1.21e-2}$&$\color{red}{1.48e-2}$\\
$2$&$0$&$3.25e-2$&$\color{blue}{2.86e-2}$&$\color{blue}{3.26e-2}$&$\color{blue}{2.48e-2}$&$\color{red}{3.13e-2}$&$\color{red}{3.17e-2}$\\
$3$&$0$&$1.94e-2$&$\color{blue}{2.66e-2}$&$\color{blue}{2.46e-2}$&$\color{blue}{2.36e-2}$&$\color{red}{1.71e-2}$&$\color{red}{2.60e-2}$\\
$4$&$0$&$Inf$&$\color{blue}{3.21e-2}$&$\color{blue}{2.83e-2}$&$\color{blue}{2.85e-2}$&$\color{red}{3.03e-2}$&$\color{red}{3.00e-2}$\\
$6$&$0$&$Inf$&$2.90e-2$&$\color{blue}{2.53e-2}$&$\color{blue}{3.24e-2}$&$\color{red}{3.46e-2}$&$\color{red}{4.10e-2}$\\
$8$&$0$&$Inf$&$4.37e-2$&$\color{blue}{3.11e-2}$&$\color{blue}{2.44e-2}$&$\color{red}{3.70e-2}$&$\color{red}{2.76e-2}$\\
$10$&$0$&$Inf$&$3.07e-2$&$\color{blue}{2.72e-2}$&$\color{blue}{3.05e-2}$&$\color{red}{3.13e-2}$&$\color{red}{2.62e-2}$\\
$12$&$0$&$Inf$&$Inf$&$3.04e-2$&$\color{blue}{2.52e-2}$&$\color{red}{2.85e-2}$&$\color{red}{3.29e-2}$\\
$14$&$0$&$Inf$&$Inf$&$3.01e-2$&$\color{blue}{3.16e-2}$&$\color{red}{3.14e-2}$&$\color{red}{3.25e-2}$\\
$16$&$0$&$Inf$&$Inf$&$3.05e-2$&$\color{blue}{3.08e-2}$&$\color{red}{3.53e-2}$&$\color{red}{2.96e-2}$\\
\hline
\end{tabular}
\end{center}
%\end{table}

%\begin{table}[htb]
\begin{center}
\small
\tabcaption{Computation time of Method 1 with $n=0$}
\label{n6ta7}
\begin{tabular}{c|c|cccccc}
\hline
n & m & $\kappa=1$ & $\kappa=5$ & $\kappa=10$ & $\kappa=20$ & $\kappa=50$ & $\kappa=100$\\
\hline
$0$&$0$&$\color{blue}{4.47e-2}$&$\color{blue}{6.61e-2}$&$\color{blue}{1.05e-1}$&$\color{blue}{1.30e-1}$&$\color{red}{2.85e-1}$&$\color{red}{Inf}$\\
$0$&$1$&$\color{blue}{4.23e-2}$&$\color{blue}{6.70e-2}$&$\color{blue}{1.06e-1}$&$\color{blue}{1.50e-1}$&$\color{red}{2.50e-1}$&$\color{red}{Inf}$\\
$0$&$2$&$3.51e-2$&$\color{blue}{6.40e-2}$&$\color{blue}{1.15e-1}$&$\color{blue}{1.78e-1}$&$\color{red}{2.46e-1}$&$\color{red}{Inf}$\\
$0$&$3$&$3.65e-2$&$\color{blue}{5.73e-2}$&$\color{blue}{9.20e-2}$&$\color{blue}{1.41e-1}$&$\color{red}{2.92e-1}$&$\color{red}{Inf}$\\
$0$&$4$&$2.84e-2$&$\color{blue}{5.02e-2}$&$\color{blue}{8.86e-2}$&$\color{blue}{1.16e-1}$&$\color{red}{2.63e-1}$&$\color{red}{Inf}$\\
$0$&$6$&$2.25e-2$&$5.27e-2$&$\color{blue}{8.67e-2}$&$\color{blue}{1.26e-1}$&$\color{red}{2.21e-1}$&$\color{red}{Inf}$\\
$0$&$8$&$1.62e-2$&$4.67e-2$&$\color{blue}{8.26e-2}$&$\color{blue}{1.38e-1}$&$\color{red}{2.31e-1}$&$\color{red}{Inf}$\\
$0$&$10$&$1.04e-2$&$3.46e-2$&$\color{blue}{6.43e-2}$&$\color{blue}{1.11e-1}$&$\color{red}{2.42e-1}$&$\color{red}{Inf}$\\
$0$&$12$&$1.45e-2$&$2.84e-2$&$7.75e-2$&$\color{blue}{1.43e-1}$&$\color{red}{2.11e-1}$&$\color{red}{Inf}$\\
$0$&$14$&$1.57e-2$&$2.56e-2$&$7.09e-2$&$\color{blue}{1.03e-1}$&$\color{red}{2.10e-1}$&$\color{red}{Inf}$\\
$0$&$16$&$1.49e-2$&$2.73e-2$&$4.96e-2$&$\color{blue}{9.85e-2}$&$\color{red}{1.99e-1}$&$\color{red}{Inf}$\\
\hline
\end{tabular}
\end{center}
%\end{table}

%\begin{table}[htb]
\begin{center}
\small
\tabcaption{Computation time of Method 2 with $n=0$}
\label{n6ta8}
\begin{tabular}{c|c|cccccc}
\hline
n & m & $\kappa=1$ & $\kappa=5$ & $\kappa=10$ & $\kappa=20$ & $\kappa=50$ & $\kappa=100$\\
\hline
$0$&$0$&$\color{blue}{1.69e-2}$&$\color{blue}{Inf}$&$\color{blue}{Inf}$&$\color{blue}{Inf}$&$\color{red}{1.37e-2}$&$\color{red}{1.76e-2}$\\
$0$&$1$&$\color{blue}{1.79e-2}$&$\color{blue}{1.71e-2}$&$\color{blue}{2.20e-2}$&$\color{blue}{1.95e-2}$&$\color{red}{1.80e-2}$&$\color{red}{1.64e-2}$\\
$0$&$2$&$Inf$&$\color{blue}{Inf}$&$\color{blue}{Inf}$&$\color{blue}{Inf}$&$\color{red}{1.41e-2}$&$\color{red}{1.81e-2}$\\
$0$&$3$&$1.78e-2$&$\color{blue}{1.62e-2}$&$\color{blue}{1.56e-2}$&$\color{blue}{1.70e-2}$&$\color{red}{2.02e-2}$&$\color{red}{1.57e-2}$\\
$0$&$4$&$Inf$&$\color{blue}{Inf}$&$\color{blue}{Inf}$&$\color{blue}{Inf}$&$\color{red}{1.69e-2}$&$\color{red}{1.58e-2}$\\
$0$&$6$&$Inf$&$Inf$&$\color{blue}{Inf}$&$\color{blue}{Inf}$&$\color{red}{1.79e-2}$&$\color{red}{1.85e-2}$\\
$0$&$8$&$Inf$&$Inf$&$\color{blue}{Inf}$&$\color{blue}{Inf}$&$\color{red}{2.02e-2}$&$\color{red}{1.91e-2}$\\
$0$&$10$&$Inf$&$Inf$&$\color{blue}{Inf}$&$\color{blue}{Inf}$&$\color{red}{1.94e-2}$&$\color{red}{1.70e-2}$\\
$0$&$12$&$Inf$&$Inf$&$Inf$&$\color{blue}{Inf}$&$\color{red}{1.58e-2}$&$\color{red}{2.02e-2}$\\
$0$&$14$&$Inf$&$Inf$&$Inf$&$\color{blue}{Inf}$&$\color{red}{1.67e-2}$&$\color{red}{1.51e-2}$\\
$0$&$16$&$Inf$&$Inf$&$Inf$&$\color{blue}{Inf}$&$\color{red}{2.10e-2}$&$\color{red}{1.91e-2}$\\
\hline
\end{tabular}
\end{center}
%\end{table}

%\begin{table}[htb]
\begin{center}
\small
\tabcaption{Computation time of Method 3 with $n=0$}
\label{n6ta9}
\begin{tabular}{c|c|cccccc}
\hline
n & m & $\kappa=1$ & $\kappa=5$ & $\kappa=10$ & $\kappa=20$ & $\kappa=50$ & $\kappa=100$\\
\hline
$0$&$0$&$\color{blue}{9.68e-3}$&$\color{blue}{6.28e-3}$&$\color{blue}{8.99e-3}$&$\color{blue}{7.99e-3}$&$\color{red}{8.32e-3}$&$\color{red}{8.87e-3}$\\
$0$&$1$&$\color{blue}{1.07e-2}$&$\color{blue}{1.30e-2}$&$\color{blue}{1.30e-2}$&$\color{blue}{1.40e-2}$&$\color{red}{1.03e-2}$&$\color{red}{1.24e-2}$\\
$0$&$2$&$3.39e-2$&$\color{blue}{2.49e-2}$&$\color{blue}{3.61e-2}$&$\color{blue}{2.49e-2}$&$\color{red}{2.73e-2}$&$\color{red}{2.80e-2}$\\
$0$&$3$&$2.62e-2$&$\color{blue}{2.53e-2}$&$\color{blue}{2.40e-2}$&$\color{blue}{3.06e-2}$&$\color{red}{2.91e-2}$&$\color{red}{2.34e-2}$\\
$0$&$4$&$3.18e-2$&$\color{blue}{3.45e-2}$&$\color{blue}{3.02e-2}$&$\color{blue}{2.97e-2}$&$\color{red}{4.06e-2}$&$\color{red}{3.59e-2}$\\
$0$&$6$&$4.06e-2$&$3.92e-2$&$\color{blue}{3.58e-2}$&$\color{blue}{4.02e-2}$&$\color{red}{4.14e-2}$&$\color{red}{4.10e-2}$\\
$0$&$8$&$4.44e-2$&$4.20e-2$&$\color{blue}{5.26e-2}$&$\color{blue}{5.54e-2}$&$\color{red}{4.35e-2}$&$\color{red}{4.53e-2}$\\
$0$&$10$&$4.94e-2$&$4.82e-2$&$\color{blue}{5.04e-2}$&$\color{blue}{5.62e-2}$&$\color{red}{5.07e-2}$&$\color{red}{4.47e-2}$\\
$0$&$12$&$6.90e-2$&$9.53e-2$&$7.15e-2$&$\color{blue}{5.96e-2}$&$\color{red}{6.99e-2}$&$\color{red}{5.90e-2}$\\
$0$&$14$&$6.04e-2$&$5.95e-2$&$7.40e-2$&$\color{blue}{6.00e-2}$&$\color{red}{7.02e-2}$&$\color{red}{7.09e-2}$\\
$0$&$16$&$9.59e-2$&$8.19e-2$&$6.90e-2$&$\color{blue}{8.19e-2}$&$\color{red}{7.29e-2}$&$\color{red}{6.87e-2}$\\
\hline
\end{tabular}
\end{center}
%\end{table}

\section*{Acknowledgment}

This work was partially supported by the National Natural Science
Foundation of China under grants 11271370.

\bibliographystyle{plain}% ¡°standard¡± styles, plain, unsrt, abbrv and alpha.
% \bibliography{Note_6}

%\end{CJK}
\end{document}